\documentclass[10pt]{amsart}
\usepackage[utf8]{inputenc}
\usepackage[T1,T2A]{fontenc}
\usepackage[english]{babel}
\usepackage{graphicx}
\usepackage{longtable}
\usepackage{wrapfig}
\usepackage{rotating}
\usepackage[normalem]{ulem}
\usepackage{amsmath}
\usepackage{amssymb}
\usepackage{capt-of}
\usepackage[bookmarks,hidelinks]{hyperref}
\usepackage{amsthm}
\usepackage{amsfonts}
\usepackage{mathtools}
\usepackage{geometry} 
\usepackage{bm}
\usepackage[warn]{mathtext}
\usepackage[bbgreekl]{mathbbol}
\usepackage{framed}
\usepackage[pdftex,dvipsnames]{xcolor}
\usepackage[shortlabels]{enumitem}
\usepackage{cancel}
\usepackage{centernot}
\usepackage{thmtools}
\usepackage{cleveref}
\usepackage{soul}
\usepackage{subfigure}

\usepackage{xargs}
\usepackage[colorinlistoftodos,prependcaption,shadow,backgroundcolor=green!25]{todonotes}
\todostyle{ref}{backgroundcolor=yellow!25}
\todostyle{think}{backgroundcolor=blue!25}
\todostyle{wrong}{backgroundcolor=red!25}
\todostyle{toadd}{backgroundcolor=cyan!25}
\todostyle{conjecture}{backgroundcolor=yellow!75}
\setuptodonotes{inline}
\makeatletter
\providecommand\@dotsep{5}
\renewcommand{\listoftodos}[1][\@todonotes@todolistname]{%
  \@starttoc{tdo}{#1}}
\makeatother

\makeatletter
\def\DKL{%
  D_{\mathrm{KL}}%
  \@ifnextchar\bgroup{\DKL@withargs}{}%
}
\def\DKL@withargs#1#2{%
  \left( #1 \, \| \, #2 \right)%
}
\def\DTV{%
  D_{\mathrm{TV}}%
  \@ifnextchar\bgroup{\DTV@withargs}{}%
}
\def\DTV@withargs#1#2{%
  \left( #1, #2 \right)%
}
\def\Df{%
  D_{f}%
  \@ifnextchar[{%
    \Df@withargs
  }{}%
}
\def\Df@withargs[#1][#2]{%
  \left( #1 \, \| \, #2 \right)%
}
\makeatother
\newcommand{\R}{\mathbb{R}}
\newcommand{\N}{\mathbb{N}}
\newcommand{\Prob}[1]{\mathbb{P} \left\{ #1 \right\}}
\newcommand{\Ind}[1]{\mathbb{1} \left\{ #1 \right\}}
\DeclareMathOperator{\erf}{erf}
\DeclareMathOperator{\erfi}{erfi}
\DeclareMathOperator{\Ei}{Ei}
\DeclareMathOperator{\Var}{Var}
\DeclareMathOperator{\clow}{c_{\mathrm{lower}}}
\DeclareMathOperator{\nlow}{n_{\mathrm{lower}}}
\DeclareMathOperator{\nup}{n_{\mathrm{upper}}}
\newcommand{\gammaEM}{\gamma_{\mathrm{EM}}}
\newcommand{\COM}[1]{}

\newcommand{\vk}{\bm}

\allowdisplaybreaks[4]

\newtheorem{theorem}{Theorem}
\newtheorem{lemma}{Lemma}
\newtheorem{proposition}{Proposition}
\newtheorem{remark}{Remark}

\newtheorem{corollary}{Corollary}

\DeclareMathOperator\diag{diag}

\makeatletter
\def\namedlabel#1#2{\begingroup
  #2%
  \def\@currentlabel{#2}%
  \phantomsection\label{#1}\endgroup
}
\makeatother
\makeatletter
\def\subsubsection{\@startsection{subsubsection}{3}%
  \z@{.5\linespacing\@plus.7\linespacing}{.1\linespacing}%
  {\normalfont\itshape}}
\def\subsection{\@startsection{subsection}{2}
  \z@{.5\linespacing\@plus.7\linespacing}{.3\linespacing}
  {\normalfont\bfseries}}
\makeatother

\address{Pavel Ievlev, Department of Actuarial Science,
  University of Lausanne\\
  UNIL-Dorigny, 1015 Lausanne, Switzerland,
  \href{https://orcid.org/0000-0003-3329-5741}{orcid.org/0000-0003-3329-5741}
}
\email{ievlev.pn@gmail.com}

\address{Timofei Shashkov, Department of Actuarial Science,
  University of Lausanne\\
  UNIL-Dorigny, 1015 Lausanne, Switzerland,
  \href{https://orcid.org/0009-0004-1567-0530}{orcid.org/0009-0004-1567-0530}
}
\email{timofei98shashkov@gmail.com}
\thanks{Supported by SNSF Grant 200021–196888.}
\author{Pavel Ievlev, Timofei Shashkov}
\date{}
\title[Upper and lower bounds on TVD and KLD between centered elliptical...]{Upper and lower bounds on TVD and KLD between centered elliptical distributions in high-dimensional setting}
\begin{document}

\begin{abstract}
In this paper, we derive some upper and lower bounds and inequalities for the
total variation distance (TVD) and the Kullback-Leibler divergence (KLD), also
known as the relative entropy, between two probability measures \( \mu \) and
\( \nu \) defined by
\begin{equation*}
  \DTV{\mu}{\nu} = \sup_{B \in \mathcal{B} (\R^n)}
  \left| \mu(B) - \nu(B) \right|
  \quad \text{and} \quad
  \DKL{\mu}{\nu} = \int_{\R^n}
  \ln \left( \frac{d\mu(x)}{d\nu(x)} \right) \, \mu(dx)
\end{equation*}
correspondingly when the dimension \( n \) is high. We begin with some
elementary bounds for centered elliptical distributions admitting densities and
showcase how these bounds may be used by estimating the TVD and KLD between
multivariate Student and multivariate normal distribution in the
high-dimensional setting. Next, we show how the same approach simplifies when we
apply it to multivariate Gamma distributions with independent components (in the
latter case, we only study the TVD, because KLD may be calculated explicitly,
see~\cite{Penny}). Our approach is motivated by the recent contribution by
Barabesi and Pratelli~\cite{BarabesiPratelli2024}.%
\end{abstract}

\maketitle



\section{Introduction}%
\label{sec:introduction}


Measuring the dissimilarity between two probability measures using total
variation distance (TVD) or Kullback-Leibler divergence (KLD) is prominent
across statistics and machine learning. In a typical application, the two
measures are either (a)~empirical and fitted, (b)~corresponding to two different
models of the same dataset or (c)~corresponding to one model fitted to two
different datasets. In case~(a), these measures are used to estimate how well
the model fits the data. For example, in hypothesis testing TVD provides
so-called minimax lower bounds on both Type I errors (false positives) and Type
II errors (false negatives) over all possible decision rules, see~\cite[Theorem
2.2]{Tsybakov2009}. In case~(b), the two measures assess whether a model may be
switched to a more preferable one within prescribed error bounds. A common
example is to switch true but intractable distribution \( \mu \) to a simple
approximation \( \nu \), such as the Laplace approximation,
see~\cite{dehaene2019,kasprzakHowGoodYour2023,Spokoiny2023a,Spokoiny2023b}, or
variational inference, see~\cite{VariationalInference2017}. In case~(c),
measures \( \DTV \) and \( \DKL \) are used to guarantee some robustness under
changes of the underlying data, such as restriction to a subset, and/or mitigate
overfitting. Since the advent of generative neural networks, the case (c) has
gained a lot of attention. Two notable examples for point~(c) are (i)~assessing
the vulnerability to data poisoning by estimating how much the fitted
distribution is affected by injecting new points into the
dataset~\cite{li2022federatedlearningbaseddefending} and (ii)~providing
quantitative privacy guarantee when the dataset used for training contains
sensitive information, such as a language model trained on private emails. In
both examples, the framework is quite similar: given a randomized algorithm
assigning to a dataset \( \mathcal{D} \) some distribution
\( \mu_{\mathcal{D}} \), check how \( \mu_{\mathcal{D}} \) differs from
\( \mu_{\mathcal{D}'} \) if \( \mathcal{D} \) and \( \mathcal{D}' \) are close.
In the differential privacy setting, the aim is to ensure that an observer
cannot infer whether a particular individual's data was part of the training set
by sampling from \( \mu_{\mathcal{D}} \) by adding as little noise to the model
as possible (utility-privacy trade-off),
see,~\cite{kaissis2023boundingdatareconstructionattacks,RassouliGunduz2020}. In
this context, TVD measures the maximum distance between the output distributions
when a single data point is removed.
\smallskip

The choice between \( \DKL \) and \( \DTV \) in specific applications is mainly
guided by (a) the necessity for a more sensitive metric (\( \DKL \) is weaker
than \( \DTV \) due to Pinsker's inequality
\( \DTV{\mu}{\nu}^2 \leq \DKL{\mu}{\nu}) / 2 \)),
see~\cite{Jia2023,GhaziIssa2024,MachineUnlearning2021}, (b) the fact \( \DKL \)
is frequently easier to compute (for many common distribution families, closed
form expressions are available, see~\cite{NielsenNock2010}), and (c)~the fact
that \( \DTV \) is usually easier to interpret. The sensitivity~(a) is
particularly important in the context of data poisoning and differential
privacy~\cite{Jia2023,GhaziIssa2024} whereas for more generic learning
alrorithms performance~(b) is of more importance. For a general overview of
probability metrics and how to choose between them, see~\cite{GibbsSu2002}.%
\smallskip

The aforementioned applications underscore the necessity of total variation and
KL divergence bounds \textit{in high-dimensional settings} where explicit
formulas are seldom available and Monte Carlo simulations may be prohibitively
costly. This area of research has attracted a lot of attention recently,
see~\cite{Devroye2023,BarabesiPratelli2024}.%
\smallskip

In this paper, we derive some elementary bounds on the TVD and KLD between
centered elliptical distributions (Propositions~\ref{tv_gen} and~\ref{kl_gen})
and them show how they may be applied when one distribution is the multivariate
Student t-distribution and the other is multivariate normal
(Theorems~\ref{main_TVD} and~\ref{main_KL}) or when both distributions are
multivariate Gamma with independent components
(Theorem\textrm{~}\ref{main_gamma}. The bounds of Theorems~\ref{main_TVD}
and~\ref{main_KL} are given in terms of standard special functions, such as the
incomplete regularized Gamma and Beta functions and the Lambert W-function.
Since the bounds work for sufficiently large \( n \), we show how to choose the
threshold (see Lemma~\ref{n0} and condition~\eqref{liminf-condition}). In
Section~\ref{sec:gamma}, we use CLT with Berry-Esseen bounds with respect to \(
n \) for the analysis of TVD between two Gamma laws in a similar way
to~\cite{BarabesiPratelli2024}. 
\smallskip

In principle, the same approach is applicable to any two elliptical
distributions for which the sets \( A_{\pm} \) from Proposition~\ref{tv_gen} may
be described in a convenient form. Moreover, the same approach may be used to
construct bounds on other divergences.
\bigskip

\textbf{Brief organization of the paper.} Section~\ref{sec:elementary_bounds} is
dedicated to some elementary bounds on the TVD and KLD between two elliptical
distributions, which we later showcase in Section~\ref{sec:mult-t-distr} by
estimating the TVD and KLD between multivariate Student and multivariate normal
distribution and in Section~\ref{sec:gamma} by estimating the TVD between two
multivariate Gamma distributions with independent components. Throughout the
paper, longer proofs are relegated to the Appendix.%
\bigskip

\textbf{Notation.}
Throughout the paper, we use the following special functions:
\begin{itemize}
  \item \( W_{0} \) and \( W_{-1} \): branches of the Lambert-W function,
  see~\cite{CorlessEtAl1996};
  \item \( \gamma ( x ; a ) = \int_0^x t^{a - 1} e^{-t} \, dt \) and \( \Gamma (
  x; a ) = \int_x^\infty t^{a-1} e^{-t} \, dt \): lower and upper incomplete
  Gamma function, see, e.g.,~\cite[Chapter 6.5]{AbramowitzStegun};
  \item \( P ( z ; a ) = \gamma ( x ; a ) / \Gamma ( a ) \): lower regularized
  incomplete Gamma function;
  \item \( B ( x ; a, b ) = \int_0^x t^{a-1} (1-t)^{b-1} \, dt \): incomplete
  Beta function;
  \item \( I ( x ; a, b ) = B ( x ; a, b ) / B ( a, b ) \): regularized
  incomplete Beta function;
  \item \( {}_{2} F_{1} \) and \( {}_{2} F_{2} \): hypergeometric functions.
  \item \( \erf ( x ) = \frac{2}{\sqrt{\pi}} \int_0^x e^{-t^2} dt \) and
  \( \erfi ( x ) = -i \erf (x) \): error function and the imaginary error
  function.
  \item \( \Ei ( x ) = \int_{-\infty}^x \frac{e^{t}}{t} \, dt \): exponential integral.
\end{itemize}
References concerning these special functions may be found in~\cite{NIST}.

\section{Elementary bounds on TVD and KLD between two elliptical distributions}%
\label{sec:elementary_bounds}

Let \(\mathcal{E} ( g, \Sigma )\) be a centered elliptical distribution
admitting a pdf given by
\begin{equation*}
  f ( x ) = \frac{1}{\sqrt{\det \Sigma}} \,
  g \left( x^\top \Sigma^{-1} x \right).
\end{equation*}
Here \( g \) is a so-called density generator. For a comprehensive treatment of
elliptical distributions, see~\cite{FangKotz1990,GuptaVarga2013}.%
\smallskip

First, we want to show that calculation of the distances between two elliptical
distributions \( \mathcal{E} ( g_1, \Sigma_1 ) \) and
\( \mathcal{E} ( g_2, \Sigma_2 ) \) may be reduced to the case when
\( \Sigma_1 = I \) and \( \Sigma_2 \) diagonal. Instead of proving it for TV and
DKL separately, we prove it for all \( f \)-divergences (see~\cite{Sason2016})
in~\Cref{covariance_reduction}. Note that a version of this lemma in a general
context is available in~\cite{Nielsen2021}.

\begin{lemma}\label{covariance_reduction}
    Let \(D\) be the eigenvalue matrix of \(\Sigma_2^{-1} \Sigma_1\). Then
    \begin{equation*}
        \Df [ \mathcal{E} ( g_1, \Sigma_1 ) ][\mathcal{E} ( g_2, \Sigma_2 ) ]
        = \Df [ \mathcal{E} ( g_1, I ) ][ \mathcal{E} ( g_2, D ) ],
    \end{equation*}
    where \( \Df \) is the \( f \)-divergence associated to a convex function \( f
    \colon [0, \infty) \to (-\infty, \infty] \) such that \( f(x) < \infty \) for all
    \( x > 0 \) and \( f(1) = 0 \) by
    \begin{equation*}
        \Df [\mu][\nu]
        = \int_{\R^n}
        f \left( \frac{d\mu(x)}{d\nu(x)} \right)
        \, d \nu(x).
    \end{equation*}
\end{lemma}

Applying this lemma with \( f(t) = t \ln t \) and \( f(t) = |t-1|/2 \) yields
the same claim for \( \DKL \) and \( \DTV \).
\smallskip

Next, we shall use the following formula for the total variation distance in
terms of the pdfs: if \( \bm{X}\sim \mathcal{E} ( g_1, I ) \) has pdf \(f_1\)
and \( \bm{Y}\sim\mathcal{E} ( g_2, D )\) has pdf \(f_2\), then
\begin{equation}
  \label{d-in-terms-of-X-and-Y}
  \DTV{\mathcal{E} (g_{1}, I)}{\mathcal{E}(g_{2}, D)}
  = \mathbb{P} \{ f_1 ( \bm{X} ) \geq f_2 ( \bm{X} ) \}
  - \mathbb{P} \{ f_1 ( \bm{Y} ) \geq f_2 ( \bm{Y} ) \}.
\end{equation}

Hereafter we assume that \(d_i = D_{ii}, i = 1,\dots, d\) and there exists
positive \(d_{-}, d_{+}\) such that \(d_{-}\leq d_{i}\leq d_{+}\) for all
\(i=1,\dots, d\).
\smallskip

The following two lemmas contain elementary bounds on TVD and DKL between
elliptical distributions.

\begin{proposition}\label{tv_gen}
  Let \(\bm{X} \sim \mathcal{E} ( g_1, I )\),
  \(\bm{Y} \sim \mathcal{E} ( g_2, D )\).
  \begin{enumerate}[(a)]
    \item If \( D = I \), then
    \begin{equation*}
      \DTV{ \mathcal{E} (g_1, I ) }{ \mathcal{E} (g_2, D) }
      = \frac{ 2 \pi^{n / 2} }{\Gamma (n / 2)}
      \int_{0}^{\infty}
      \mathbb{1} \{ r^2 \in A\}
      \left( g_1 (r^2) - g_2 (r^2) \right)
      \, r^{n-1}
      \, dr
    \end{equation*}
    with \( A = \{ t \in [0,  \infty) \colon g_1(t) \geq g_2(t) \} \).
    \item If \( g_1 \) and \( g_2 \) are decreasing on
          \( [0, \infty) \), then
          \begin{align*}
            &
            \frac{2 \pi^{n / 2}}{\Gamma (n / 2)}
            \int_0^{\infty}
            \Big[
            g_{1} (r^{2}) \, \mathbb{1} \{ r^{2} \in A_{-} \}
            - g_{2} (r^{2}) \, \mathbb{1} \{ d_{-} r^{2} \in A_{+} \}
            \Big]
              r^{n-1} \, dr
            \\[7pt]
            & \hspace{0.5cm}
              \leq \DTV{ \mathcal{E} (g_{1}, I) }{ \mathcal{E}(g_{2}, D) }
            \\[7pt]
            & \hspace{0.5cm}
              \leq
              \frac{2 \pi^{n / 2}}{\Gamma (n / 2)}
              \int_0^{\infty}
              \Big[
              g_{1} (r^{2}) \, \mathbb{1} \{ r^{2} \in A_{+} \}
              - g_{2} (r^{2}) \, \mathbb{1} \{ d_{+} r^{2} \in A_{-} \}
              \Big]
              r^{n-1} \, dr
          \end{align*}
    with
    \begin{equation}
      \label{def:Apm}
      A_\pm = \left\{
        t \in [0, \infty) :
        \prod_{i = 1}^n d_i^{1 / 2} g_1(t)
        \geq g_2 \left( \frac{t}{d_\mp} \right)
      \right\}.
    \end{equation}
  \end{enumerate}
\end{proposition}

\begin{proposition}\label{kl_gen}
  Let \(\vk X\sim \mathcal{E} ( g_1, I ) \) and
  \(\vk Y\sim \mathcal{E} ( g_2, D )\).
  \begin{enumerate}[(a)]
    \item If \(D=I\), then
    \begin{equation*}
      \DKL{ \mathcal{E}(g_{1},I) }{ \mathcal{E} (g_{2},D) }
      = \frac{2\pi^{n/2}}{\Gamma(n/2)}
      \int_{0}^{\infty}
      \ln\left(
        \frac{g_1(r^2)}{g_2(r^2)}
      \right) \, g_1(r^2) \,
      r^{n-1} \, dr.
    \end{equation*}
    \item If \( g_2 \) is decreasing on \( [0, \infty) \), then:
    \begin{align*}
      &
      \frac{2\pi^{n/2}}{\Gamma(n/2)}
      \int_{0}^{\infty} g_1(r^2)
      \ln\left(
        \frac{g_1(r^2)}{g_2\left(r^2/d_{+}\right)}
      \right)
      \, r^{n-1} \, dr
      +\frac{1}{2} \sum_{i=1}^{n}
      \ln d_i
      \\[7pt]
      & \hspace{0.5cm}
      \leq \DKL{ \mathcal{E}(g_1,I) }{ \mathcal{E}(g_2,D) }
      \\[7pt]
      & \hspace{0.5cm}
      \leq \frac{2\pi^{n/2}}{\Gamma(n/2)}
      \int_{0}^{\infty}g_1(r^2)
      \ln\left(
        \frac{g_1(r^2)}{g_2\left(r^2/d_{-}\right)}
      \right)
      \, r^{n-1} \, dr
      +\frac{1}{2} \sum_{i=1}^{n} \ln d_i,
    \end{align*}
  \end{enumerate}
\end{proposition}

\section{Multivariate Student t-distribution against multivariate normal distribution}%
\label{sec:mult-t-distr}
The pdfs of the centered multivariate \( t \)-distribution with \( \Sigma = I \)
and of the multivariate normal distribution with
\( \Sigma = D = \diag ( d_i ) \) are given by
\begin{equation}\label{def:pdfs}
  f_1 ( x ) =
  \frac{\Gamma ( ( \nu + n ) / 2 )}{\Gamma ( \nu / 2 ) \, \nu^{n/2} \pi^{n / 2}}
  \left[ 1 + \frac{1}{\nu} \sum_{i = 1}^n x_i^2 \right]^{-( \nu + n ) / 2}
  \quad \text{and} \quad
  f_2 ( x ) =
  \frac{1}{( 2 \pi )^{n/2} \prod_{i = 1}^n d_i^{1/2} }
  \exp \left( -\frac{1}{2} \sum_{i = 1}^n \frac{x_i^2}{d_i} \right).
\end{equation}

In this section, we apply~\Cref{tv_gen} and~\Cref{kl_gen} to obtain lower and upper
bounds on the total variation distance and KL divergence between these two
distributions.

\subsection{Total variation distance}%
\label{sec:main_TV}

To apply \Cref{tv_gen}, we need first to find a convenient description of the
sets
\[
  A_{\pm}=\left\{t\in[0,\infty) :
    \left(\prod_{i=1}^{n}d_i\right)^{1/2}
    \frac{\Gamma ( ( \nu + n ) / 2 )}{\Gamma ( \nu / 2 ) \, \nu^{n/2} \pi^{n / 2}}
    \left[ 1 + \frac{t}{\nu} \right]^{-( \nu + n ) / 2}\geq
    \frac{1}{(2\pi)^{n/2}}e^{-\frac{t}{2d_{\mp}}}\right\}.
\]
Rewriting \( A_{\pm}\) as follows
\begin{equation*}
  A_{\pm}
  =
  \left\{
    x \colon
    \left( 1 + \frac{t}{\nu} \right)^{( \nu + n ) / 2}
    \exp \left( -\frac{t}{2d_{\mp}} \right)
    \leq
    \frac{\Gamma ( ( \nu + n ) / 2 ) 
    \cdot ( 2 \pi )^{n/2}}{ \Gamma ( \nu / 2 ) \nu^{n/2} \pi^{n/2} }
    \prod_{i = 1}^n d_i^{1/2}
  \right\}
\end{equation*}
and denoting
\begin{equation}
  \label{c-def}
  c =
  \frac{\Gamma ( ( \nu + n ) / 2 ) \cdot ( 2 \pi )^{n/2}}{ \Gamma ( \nu / 2 ) \nu^{n/2} \pi^{n/2} }
  \prod_{i = 1}^n d_i^{1/2}
  = \frac{\Gamma ( ( \nu + n ) / 2 )}{ \Gamma ( \nu / 2 ) } \left( \frac{2}{\nu} \right)^{n/2}
  \prod_{i = 1}^n d_i^{1/2},
\end{equation}
we find that both sets \( A_{+} \) and \(A_{-}\) are sublevel sets of a function
\( \varphi_{\alpha, \gamma} ( x ) = (1 + x)^{\alpha} e^{-\gamma x} \):
\begin{equation}
  \label{A-pm-def}
  A_{\pm}
  = \left\{
    t\colon
    \varphi_{\alpha, \gamma_{\pm}}
    \left( \frac{t}{\nu}  \right) \leq c
  \right\},
  \quad \text{where} \quad
  \alpha = \frac{\nu + n}{2}, \quad \gamma_{\pm} = \frac{\nu}{2 d_{\mp}}.
\end{equation}
In the following lemma we simplify the description of \( A_{\pm} \) further.%

\begin{lemma}\label{sublevel-lemma}
  Let \( \varphi_{\alpha, \gamma} ( x ) = ( 1 + x )^{\alpha} e^{-\gamma x} \)
  and \( \alpha, \, \gamma, \, c > 0 \). Then there exist two numbers
  \( a_{\alpha, \gamma} ( c ) < b_{\alpha, \gamma} ( c ) \) such that the
  sublevel set
  \begin{equation*}
    D_{\alpha, \gamma} ( c ) \coloneqq
    \{ x \geq 0 \colon \varphi_{\alpha, \gamma} ( x ) \leq c \}
  \end{equation*}
  admits the following representation:
  \begin{equation}
    \label{eq:D-representation}
    D_{\alpha, \gamma} ( c )
    = \left\{
      x \geq 0 \colon
      x \leq a_{\alpha, \gamma} ( c ) \ \text{or} \ x \geq b_{\alpha, \gamma} ( c )
    \right\}.
  \end{equation}
  Moreover,
  \begin{itemize}
    \item If either of the following cases occurs,
    \begin{itemize}
      \item[(a)] \( \gamma > \alpha \),
      \item[(b)] \( \gamma \leq \alpha \) and \( 0 < c < 1 \),
    \end{itemize}
    then \( a_{\alpha, \gamma} ( c ) < 0 \), and therefore
    \( D_{\alpha, \gamma} ( c ) = [ b_{\alpha, \gamma} ( c ) , \infty) \).
    \item If either of the following cases occurs,
    \begin{itemize}
      \item[(c)] \( \gamma > \alpha \) and \( c \geq 1 \),
      \item[(d)] \( \gamma \leq \alpha \) and \( c \geq e^{\gamma} ( \alpha / \gamma e )^{\alpha} \),
    \end{itemize}
    then \( b_{\alpha, \gamma} ( c ) \leq 0 \), and therefore
    \( D_{\alpha, \gamma} ( c ) = [0, \infty) \).
    \item If \( \gamma \leq \alpha \) and
    \( 1 \leq c < e^{\gamma} ( \alpha / \gamma e )^{\alpha} \), then
    \( a_{\alpha, \gamma} ( c ) > 0 \) and therefore
    \( D_{\alpha, \gamma} ( c ) = [0, a_{\alpha, \gamma} ( c )] \cup [b_{\alpha, \gamma} ( c ), \infty) \).
  \end{itemize}
\end{lemma}

\begin{remark}\label{Lambert}
  The aforementioned numbers \(a_{\alpha, \gamma} ( c )\) and
  \(b_{\alpha, \gamma} ( c )\) can be calculated explicitly using branches
  \(W_{-1}\) and \(W_{0}\) of the Lambert function:
  \begin{equation}\label{a b}
    a_{\alpha, \gamma} ( c )
    = -1 -\frac{\alpha}{\gamma} \,
    W_0 \left( -\frac{\gamma c^{1/\alpha} e^{-\gamma / \alpha} }{\alpha} \right)
    \ \ \ \text{ and } \ \ \
    b_{\alpha, \gamma} ( c )
    = -1 -\frac{\alpha}{\gamma} \,
    W_{-1} \left( -\frac{\gamma c^{1/\alpha} e^{-\gamma / \alpha} }{\alpha} \right).
  \end{equation}
\end{remark}

Consider the following function:
\begin{equation}\label{c_down}
  \clow(n)
  = \frac{\Gamma ( ( \nu + n ) / 2 )}{ \Gamma ( \nu / 2 ) }
  \left( \frac{2}{\nu} \right)^{n/2}d_{-}^{n/2}.
\end{equation}
Clearly, \( c \geq \clow(n) \) for all \(n\in\mathbb{N}\). The following three
lemmas study the behaviour of the function \( \clow(n) \).
\begin{lemma}\label{ndown}%
  If \( \clow(1) \geq 1\), then \( \clow \) is strictly increasing as a function
  on \(\N\).
\end{lemma}


\begin{lemma}\label{n0}%
  Let $n_0 = 2 k$, where $k$ satisfies
  \begin{equation}
    \label{k-condition}
    \frac{1}{k} \sum_{l = 0}^{k-1}
    \ln \left( 1 + \frac{2 l}{\nu} \right)
    \geq -\ln d_-.
  \end{equation}
  Then $\clow(n_0) \geq 1$.
\end{lemma}
\begin{proof}[Proof of \Cref{n0}]
  Let $n_0 = 2k$. By~\eqref{c_down},
  \begin{equation*}
    \clow(n_0)
    = \frac{\Gamma \left( \frac{\nu}{2} + k \right)}{\Gamma \left( \frac{\nu}{2}  \right)}
    \left( \frac{2}{\nu} \right)^{k}
    d_{-}^{k}
    =
    \frac{2^k}{\nu^k}
    \prod_{l = 0}^{k-1} \left( \frac{\nu}{2} + l \right)
    d_{-}^{k}
    =
    \prod_{l = 0}^{k-1} \left( 1 + \frac{2l}{\nu} \right)
    d_{-}^{k}.
  \end{equation*}
  Hence, if $k$ satisfies~\eqref{k-condition}, we have
  \begin{equation*}
    \ln \clow(n_0)
    = k \ln d_-
    + \sum_{l = 0}^{k-1} \ln \left( 1 + \frac{2l}{\nu} \right)
    \geq 0
    \iff 
    \clow(n_0) \geq 1.
  \end{equation*}
\end{proof}

\begin{remark}
  Note that since the sequence
  $k^{-1} \sum_{l = 0}^{k-1} \ln ( 1 + 2l / \nu ) \sim \ln k\to\infty$ as
  $k \to \infty$, such $n_0$ always exists.
\end{remark}

Note that if \(d_{-}\geq 1\), then one can check that \(\clow(1)\geq 1\) and by
\Cref{ndown} we obtain that \(\clow(n) \geq 1\) for all \(n\in\N\), and, in
particular, \(n_0=2\). The next lemma discusses the case \(d_{-}<1\).

\begin{lemma}\label{minimum}%
  If \(d_{-}< 1\), then the function \(\clow(x)\) has one minimum
  \(x_{\min}>\frac{\nu}{d_{-}}-\nu\) in \([0,\infty)\) and, moreover,
  \(\clow(x)\) strictly decreases as \(x<x_{\min}\) and strictly increases as
  \(x>x_{\min}\).
\end{lemma}

\begin{corollary}\label{cor}%
  \(n_{0}> \nu/d_{-}-\nu\), where \(n_0\) is as in \Cref{n0}.
\end{corollary}

\begin{proof}[Proof of \Cref{cor}]
  Assume the opposite. Since \(n_0\geq 2\), using \Cref{minimum} we conclude
  that \(x_{\min}\geq \nu/d_{-}-\nu\geq 2\) and, therefore, by \Cref{n0}
  \[
    \clow(1) > \clow(2) \geq \clow(n_0) \geq 1.
  \]
  Hence by \Cref{ndown}, we conclude that \( \clow(1) < \clow(2) \), which
  contradicts the inequality above.
\end{proof}
\begin{lemma}\label{n1}%
  \(c<e^{\gamma_{-}} ( \alpha / \gamma_{-} e )^{\alpha}\) for all \(n\geq 1.\)
\end{lemma}

\begin{remark}
  It turns out that without any additional information on \(d_i\) both cases
  \(c\geq e^{\gamma_{+}} ( \alpha / \gamma_{+} e )^{\alpha}\) and
  \(c<e^{\gamma_{+}} ( \alpha / \gamma_{+} e )^{\alpha}\) are possible even for
  large \(n\). Indeed, one can check that
  \[
    \frac{c}{e^{\gamma_{-} } 
    ( \alpha / \gamma_{-}  e )^{\alpha}}
    \sim 
    Cn^{-1/2}
    \prod_{i=1}^{n}
    \left(\frac{d_{i}}{d_{-}}\right)^{1/2}
  \]
  for some positive constant \(C\) as \(n\to\infty\). Therefore, if we take, for
  example, a sequence \(d_{1}=1, d_{2}=\dots=d_{n}=2\), then for large \(n\) we
  have that \(c> e^{\gamma_{+}} ( \alpha / \gamma_{+} e )^{\alpha}\). On the
  other hand, if \(d_{1}=\dots=d_{n}=d_{-}=d_{+}\), then \(c<e^{\gamma_{+}} (
  \alpha / \gamma_{+} e )^{\alpha}\) for \(n\) large enough.
\end{remark}

Hereafter we assume \(n_0\) to be the minimal possible integer, which satisfies
conditions of the \Cref{n0}. Now we discuss an important corollary from
\Cref{sublevel-lemma}, \Cref{ndown}, \Cref{n0}, \Cref{minimum}, \Cref{n1} and
\Cref{cor}.

\begin{corollary}\label{corollary}
  Under conditions above for all \(n\geq n_0\) and \( ( d_i, \ i \geq 1 ) \)
  holds
  \[
    D_{\alpha, \gamma_-} ( c )
    = [ 0, a_{\alpha, \gamma_-} ( c ) ] \cup [ b_{\alpha, \gamma_-} ( c ), \infty ),
    \quad 0 < a_{\alpha, \gamma_-} ( c ) < b_{\alpha, \gamma_-} ( c ).
  \]
  If there exists \(\nup >0\) such that for \(n\geq \nup\)
  \begin{equation}\label{liminf-condition}
    c< e^{\gamma_{+}} ( \alpha / \gamma_{+} e )^{\alpha},
  \end{equation}
  then for all \( n\geq \max(n_0, \nup )\) holds
  \begin{equation*}
    D_{\alpha, \gamma_+} ( c )
    = [ 0, a_{\alpha, \gamma_+} ( c ) ] \cup [ b_{\alpha, \gamma_+} ( c ), \infty),
    \quad 0 < a_{\alpha, \gamma_+} ( c ) < b_{\alpha, \gamma_+} ( c ).
  \end{equation*}
  If, on the other hand, there exists \(\nlow\) such that for \(n\geq
  \nlow\)
  \begin{equation}\label{limsup-condition}
    c \geq e^{\gamma_{+}} ( \alpha / \gamma_{+} e )^{\alpha},
  \end{equation}
  then for all \( n\geq \text{max}(\nlow, n_0)\)
  \begin{equation*}
    D_{\alpha, \gamma_+} ( c ) = [0, \infty).
  \end{equation*}
\end{corollary}
\begin{proof}
  By \Cref{sublevel-lemma}, it is sufficient to check that \(c< e^{\gamma_{-}} (
  \alpha / \gamma_{-} e )^{\alpha}\), \(\alpha>\gamma_\pm\) and \(c>1\) for all
  \(n\geq n_0\). By \Cref{n1}, we have that \(c< e^{\gamma_{-}} ( \alpha /
  \gamma_{-} e )^{\alpha}\) for all \(n\in\N\). One can check that
  \(\alpha>\gamma_{\pm}\iff n>\nu/d_{\mp}-\nu\). By \Cref{cor}, we obtain that
  \(n\geq n_0> \nu/d_{-}-\nu\geq\nu/d_{+}-\nu.\) Hence, \(\alpha>\gamma_\pm\)
  for all \(n\geq n_0.\)

  Finally, if \(\clow(1)\geq 1\), then by \Cref{ndown} \(c(n)>1\) for all \(n\).
  If \(\clow(1) < 1\), then from \Cref{n0} and \Cref{minimum} it follows that
  \(\clow(n)\geq 1\) for all \(n\geq n_0\). Thus, the proof follows.
\end{proof}

From \Cref{corollary} it follows that
\( A_{ - } \): by~\eqref{eq:D-representation},
\begin{equation*}
  A_- =
  \left\{
    t\in[0,\infty)\colon
    t
    \leq \nu a_{\alpha, \gamma_-} ( c )
  \right\}
  \cup
  \left\{
    t\in[0,\infty)\colon
    t
    \geq \nu b_{\alpha, \gamma_-} ( c )
  \right\}.
\end{equation*}
Similarly, for \( A_+ \) we have
\begin{equation*}
  A_+ =
  \left\{
    t\in[0,\infty)\colon
    t
    \leq \nu a_{\alpha, \gamma_+} ( c )
  \right\}
  \cup
  \left\{
    t\in[0,\infty)\colon
    t
    \geq \nu b_{\alpha, \gamma_+} ( c )
  \right\},
\end{equation*}
if the condition~\eqref{liminf-condition} is satisfied and
\( A_+ = [0,\infty)\), if the condition~\eqref{limsup-condition} is satisfied.
Combining this description with \Cref{tv_gen}, we obtain bounds for the total
variation distance.

\begin{theorem}\label{main_TVD} Let \( \mathcal{E} ( g_1, I ) \) and \(
  \mathcal{E} ( g_2, D ) \) be the multivariate centered t-distribution and the
  multivariate centered normal distribution defined by their
  densities~\eqref{def:pdfs}.
  \begin{enumerate}[1.]
    \item If \Cref{limsup-condition} is satisfied, then for \(n\geq \max(n_0,
    \nlow)\)
    \[
      \DTV{ \mathcal{E}(g_1,I) }{ \mathcal{E}(g_2,D) }
      \leq
      P \left( \frac{\nu b_{\alpha, \gamma_-}}{2 d_+} ; \frac{n}{2} \right)
      - P \left( \frac{\nu a_{\alpha, \gamma_-}}{2 d_+} ; \frac{n}{2} \right) 
    \]
    \item If~\eqref{liminf-condition} is satisfied, then for
    \(n\geq\max(n_0, \nup )\) there are the following bounds for
    \(\DTV{ \mathcal{E}(g_1,I)}{\mathcal{E}(g_2,D)}\):
    \begin{align*}
      &
      \DTV{ \mathcal{E}(g_1,I) }{ \mathcal{E}(g_2,D) }
      \\[7pt]
      & \hspace{1cm}
        \leq
        P \left( \frac{\nu b_{\alpha, \gamma_-}}{2 d_+} ; \frac{n}{2} \right)
        - P \left( \frac{\nu a_{\alpha, \gamma_-}}{2 d_+} ; \frac{n}{2} \right) 
      - I \left( \frac{1}{1 + \nu b_{\alpha, \gamma_+}^2} ; \frac{\nu}{2}, \frac{n}{2} \right) 
      + I \left( \frac{1}{1 + \nu a_{\alpha, \gamma_+}^2} ; \frac{\nu}{2}, \frac{n}{2} \right) 
    \end{align*}
    and
    \begin{align*}
      &
      \DTV{ \mathcal{E}(g_1,I) }{ \mathcal{E}(g_2,D) }
      \\[7pt]
      & \hspace{1cm} \geq
        P \left( \frac{\nu b_{\alpha, \gamma_+}}{2 d_-} ; \frac{n}{2} \right) 
        - P \left( \frac{\nu a_{\alpha, \gamma_+}}{2 d_-} ; \frac{n}{2} \right) 
        - I \left( \frac{1}{1 + \nu b^2_{\alpha, \gamma_-}} ; \frac{\nu}{2} ; \frac{n}{2} \right) 
        + I \left( \frac{1}{1 + \nu a^2_{\alpha, \gamma_-}} ; \frac{\nu}{2} ; \frac{n}{2} \right) 
    \end{align*}
    where \( P \) is the lower regularized incomplete Gamma function and \( I \)
    is the regularized incomplete Beta function.
  \end{enumerate}
\end{theorem}

\begin{remark}
  Since \( x \mapsto I ( x ; a, b ) \) is increasing, the upper bound in point
  (2) of the last theorem is lower than that of the point (1).
\end{remark}

\begin{figure}
  \centering
  \subfigure{\includegraphics[width = 0.3\textwidth]{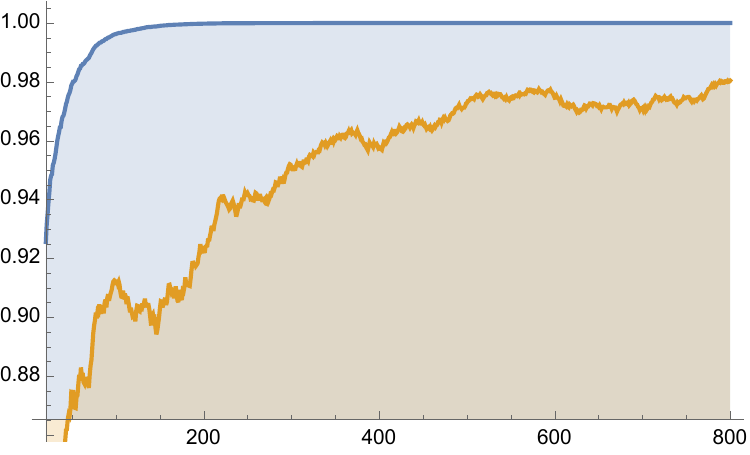}} 
  \hspace{0.01\textwidth}
  \subfigure{\includegraphics[width = 0.3\textwidth]{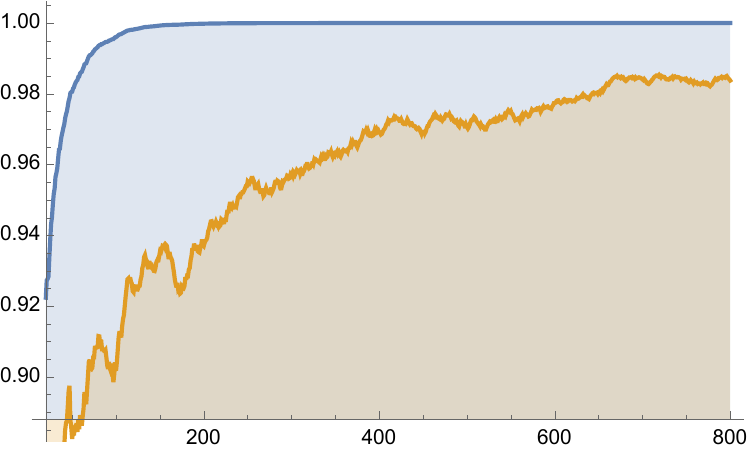}} 
  \hspace{0.01\textwidth}
  \subfigure{\includegraphics[width = 0.3\textwidth]{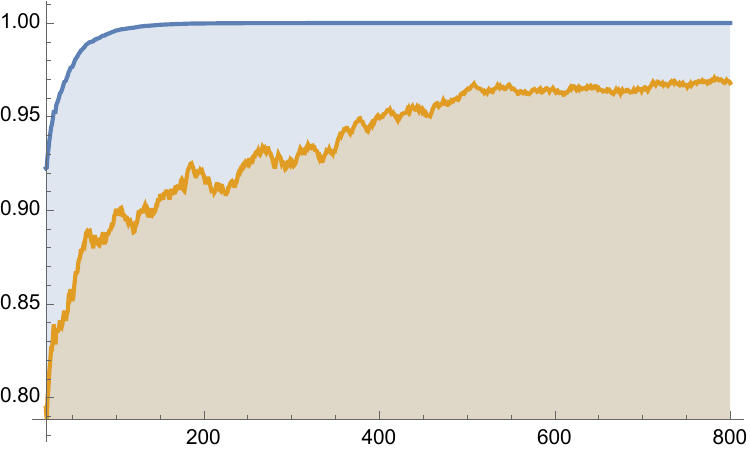}} 
  \caption{These plots illustrate the lower and upper bounds behaviour for three different vectors \( d = (d_i)_{i = 1, \dots , 800} \), randomly chosen from \( [0.95, 1.01]^{800}  \).}
\end{figure}


\subsection{KL divergence}%
\label{subsec:main_KL}

Before stating the next theorem, we introduce auxiliary notation that will be
used in its formulation. For \(k \in \N\) such that \(0<k<n/2\) define
\begin{equation}\label{alpha k n}
  \alpha_{\mp,n-2}
  = 2^{1-\frac{n}{2}}
  e^{\frac{\nu}{2 d_{\mp}}}
  \Gamma \left(\frac{n}{2}\right)
  \left(\frac{d_{\mp}}{\nu}\right)^{1-\frac{n}{2}}
  \Gamma \left(1-\frac{n}{2},\frac{\nu}{2 d_{\mp}}\right)
\end{equation}
and let for \( n \in \N \)
\begin{equation}\label{Delta mp n}
  \Delta_{\mp, n}
  = \begin{dcases}
    \sqrt{ \pi } \left( 
      \pi \erfi \left( \sqrt{ \frac{\nu}{2 d_{\mp}} } \right) 
      + \ln \left( \frac{d_{\mp}}{2 \nu} \right)
      - \gammaEM
    \right) 
    - \frac{\sqrt{ \pi } \nu }{d_\mp}
    {}_2 F_2 \left( 1, 1 ; \frac{3}{2}, 2; \frac{\nu}{2 d_{\mp}} \right) 
    & \text{if} \ n = 1,
    \\
    -e^{\frac{\nu}{2 d_{\mp}}}
    \text{Ei}\left(-\frac{\nu}{2d_{\mp}}\right)
    & \text{if} \ n = 2,
    \\
    \Gamma(n/2)\Delta_{\mp, 2}
    +\Gamma(n/2)\sum_{l=1}^{n/2-1}
    \frac{\alpha_{\mp,n-2l}}{\Gamma(n/2-l+1)}
    & \text{if} \ n \mid 2, n > 2,
    \\
    \Gamma(n/2)\Delta_{\mp, 2}
    +\Gamma(n/2)\sum_{l=1}^{n/2-1}
    \frac{\alpha_{\mp,n-2l}}{\Gamma(n/2-l+1)}
    & \text{if} \ n \nmid 2, n > 1.
  \end{dcases}
\end{equation}
Here \( \gammaEM \) is the Euler-Mascheroni constant.

\begin{theorem}\label{main_KL}
  Under assumptions given in this section,
  \begin{enumerate}[1.]
    \item If \(\nu>2\), then
    \begin{align*}
      &
        \DKL{\mathcal{E}(g_1,I)}{\mathcal{E}(g_2,D)}
      \\[7pt]
      & \hspace{1cm}
        =\frac{1}{2} \sum_{i=1}^{n} \ln d_i
        + \frac{\nu}{\nu-2}\sum_{i=1}^{n}\frac{1}{2d_{i}}
        -\frac{1}{2} (n+\nu)
        \left(
        \psi_{0}\left(\frac{n+\nu}{2}\right)
        - \psi_{0}\left(\frac{\nu}{2}\right)
        \right)
        +\ln\left(
        \frac{\Gamma ( ( \nu + n ) / 2 )}{\Gamma ( \nu / 2 ) \, \pi^{n / 2}\nu^{n/2}}
        \right).
    \end{align*}

    \item If \(\nu\leq 2\), then \( \DKL{ \mathcal{E}(g_1,I) }{
    \mathcal{E}(g_2,D) } = \infty\).

    \item For all \( \nu > 0 \), we have
    \begin{align*}
      &
      \max\left(
        0,
        \frac{\nu+n}{2\Gamma(n/2)} \Delta_{-, n}
        -\ln\left(
          \frac{\Gamma ( ( \nu + n ) / 2 )}{\Gamma ( \nu / 2 ) \, \pi^{n / 2}\nu^{n/2}}
        \right)
        - \frac{n}{2}
        - \frac{n}{2}\ln(2\pi)
        -\frac{1}{2}\sum_{i=1}^{n}\ln d_i
      \right)
      \\[7pt]
      & \hspace{1cm}
        \leq
        \DKL{ \mathcal{E}(g_2,D) }{ \mathcal{E}(g_1,I) }
      \\[7pt]
      & \hspace{1cm}
        \leq
        \frac{\nu+n}{2\Gamma(n/2)}\Delta_{+, n}
        - \ln\left(
        \frac{\Gamma ( ( \nu + n ) / 2 )}{\Gamma ( \nu / 2 ) \, \pi^{n / 2}\nu^{n/2}}
        \right)
        -\frac{n}{2}
        -\frac{n}{2}\ln(2\pi)
        -\frac{1}{2}\sum_{i=1}^{n}\ln d_i,
    \end{align*}
  \end{enumerate}
  where \(\Delta_{\mp,n}\) is given in~\eqref{Delta mp n}.
\end{theorem}


\section{Distances between two Gamma distributions}%
\label{sec:gamma}

In this section we derive bounds on the total variation distance between two
multivariate Gamma distributions with independent components defined by their
pdfs
\begin{equation*}
  f_1 ( x )
  = \prod_{i = 1}^n
  \frac{
    \lambda_i^{\alpha_i} x_i^{\alpha_i-1} e^{-\lambda_i x_i}
  }{\Gamma ( \alpha_i )}
  \quad \text{and} \quad
  f_2 ( x )
  = \prod_{i = 1}^n
  \frac{
    \mu_i^{\beta_i} x_i^{\beta_i-1} e^{-\mu_i x_i}
  }{\Gamma ( \beta_i )},
\end{equation*}
\(\alpha_i, \beta_i, \lambda_i, \mu_i>0\) for \(i=1,\dots, n\). Before we
formulate our next result, recall the Berry Esseen inequality.

Let \(Z_i, {i=1,2,\dots }\) be a sequence of independent random variables such
that
\begin{equation}\label{student}
  \mathbb{E}\{Z_i\}=0, \ \ \mathbb{E}\{Z_i^2\}=\sigma_i^2
  \ \text{ and } \  \mathbb{E}\{|Z_i|^3\}=\rho_i<\infty, \ \ \ i=1,2,\dots.
\end{equation}
Then according to Berry-Esseen theorem there exists an absolute constant \(C_0\)
such that
\begin{equation}\label{berry-esseen}
  \sup_{x\in\mathbb{R}}|F_{n}(x)-\Phi(x)|\leq C_0\kappa_{n}
\end{equation}
where \(F_{n}\) and \(\Phi\) are distribution functions of
\(S_n=(Z_1+\dots+Z_n)/\left(\sqrt{\sigma_1^2+\dots+\sigma_n^2}\right)\) and
\(Z\sim N(0,1)\) respectively, and
\(\kappa_{n}=(\sum_{i=1}^{n}\rho_i)/(\sum_{i=1}^{n}\sigma_i^2)^{3/2}.\)

Denote distributions of \(\vk X\) and \(\vk Y\) with pdfs \(f_1\) and \(f_2\)
respectively by \(\mathbb{P}_{\vk X}\) and \(\mathbb{P}_{\vk Y}\). Moreover, in
this section we assume \(\psi_0\) and \(\psi_1\) to be polygamma functions of
order \(0\) and \(1\).

\begin{theorem}\label{main_gamma}
  \[
    \DTV{ \mathbb{P}_{\bm{X}} }{ \mathbb{P}_{\bm{Y}} }
    = 
    \mathbb{P} \left\{
      \frac{
        c - \sum_{i = 1}^n \mathbb{E} \{ Z_i \}
      }{
        \sqrt{\sum_{i = 1}^n \Var \{ Z_i \}}
      }
      \leq Z \leq
      \frac{c - \sum_{i = 1}^n \mathbb{E} \{ \widetilde{Z}_i \}}{
        \sqrt{\sum_{i = 1}^n \Var \{ \widetilde{Z}_i \}}
      }
    \right\}
    + \varepsilon,
  \]
  where
  \begin{align*}
    Z_i &= ( \alpha_i - \beta_i ) \ln X_i + ( \mu_i - \lambda_i ) X_i,
          \quad X_i \sim \operatorname{Gamma} ( \alpha_i, \lambda_i ),
    \\
    \widetilde{Z}_i & = ( \alpha_i - \beta_i ) \ln Y_i + ( \mu_i - \lambda_i ) Y_i,
                      \quad Y_i \sim \operatorname{Gamma} ( \beta_i, \mu_i ),
                      \ \ Z\sim N(0,1),
    \\ \mathbb{E} \{ Z_i \}
        & = ( \alpha_i - \beta_i )
          \Big( \psi_0 ( \alpha_i ) - \ln \lambda_i \Big)
          + \alpha_i \left( \frac{\mu_i}{\lambda_i} - 1 \right),
    \\[7pt]
    \Var \{ Z_i \}
        & =
          (\alpha_i -\beta_i )^2 \psi_1 (\alpha_i )
          -\frac{(\lambda_i -\mu_i ) (\alpha_i  (\lambda_i +\mu_i )
          -2 \beta_i  \lambda_i
          )}{\lambda_i^2},
    \\ \mathbb{E} \{ \widetilde{Z}_i \}
        & = ( \alpha_i - \beta_i )
          \left(
          \psi_0 ( \beta_i ) - \ln \mu_i
          \right)
          + \beta_i \left( 1 - \frac{\lambda_i}{\mu_i} \right),
    \\[7pt]
    \Var \{ \widetilde{Z}_i \}
        & = (\alpha_i -\beta_i )^2 \psi_1(\beta_i )
          + \frac{(\lambda_i -\mu_i ) (\beta_i  (\lambda_i +\mu_i )-2 \alpha_i  \mu_i )}{\mu_i ^2},
  \end{align*}
  \[
    |\varepsilon|\leq C_0(\kappa_{n}+\widetilde{\kappa}_{n}), \ \ \ \kappa_n=\frac{\sum_{i=1}^{n}\rho_i}{(\sum_{i=1}^{n} \Var \{ Z_i \})^{3/2}} \ \ \text{ and } \ \ \widetilde{\kappa}_n=\frac{\sum_{i=1}^{n}\widetilde{\rho}_i}{(\sum_{i=1}^{n} \Var \{ \widetilde{Z}_i \})^{3/2}}
  \]
  and
  \[
    \rho_i=\mathbb{E}\{|Z_i-\mathbb{E}\{Z_i\}|^3\} \ \ \text{ and } \ \ \widetilde{\rho}_i=\mathbb{E}\{|\widetilde{Z}_i-\mathbb{E}\{\widetilde{Z}_i\}|^3\}.
  \]
\end{theorem}
\begin{proof}
  As in \Cref{sec:mult-t-distr}, we shall begin by describing the set
  $\{ x \colon f_1(x) \geq f_2(x) \}$:
  \begin{align}
    \{ x \colon f_1 ( x ) \geq f_2 ( x ) \}
    & =
      \left\{
      x \colon
      \prod_{i = 1}^n
      x_i^{\alpha_i - \beta_i} e^{(\mu_i - \lambda_i) x_i}
      \geq
      \prod_{i = 1}^n
      \frac{
      \Gamma ( \alpha_i )
      \mu_i^{\beta_i}
      }{
      \Gamma ( \beta_i )
      \lambda_i^{\alpha_i}
      }
      \right\}
    \\[7pt]
    & =
      \left\{
      x \colon
      \sum_{i = 1}^n
      \Big(
      ( \alpha_i - \beta_i ) \ln x_i
      + ( \mu_i - \lambda_i ) x_i
      \Big)
      \geq c
      \right\},
  \end{align}
  where
  \begin{equation*}
    c =
    \sum_{i = 1}^n
    \ln \left(
      \frac{
        \Gamma ( \alpha_i ) \,
        \mu_i^{\beta_i}
      }{
        \Gamma ( \beta_i ) \,
        \lambda_i^{\alpha_i}
      }
    \right).
  \end{equation*}
  Hence
  \begin{equation*}
    \DTV{ \mathbb{P}_{\bm{X}} }{ \mathbb{P}_{\bm{Y}} }
    = \mathbb{P} \left\{
      \sum_{i = 1}^n Z_i
      \geq c
    \right\}
    - \mathbb{P} \left\{
      \sum_{i = 1}^n \widetilde{Z}_i
      \geq c
    \right\}
  \end{equation*}
  and therefore
  \begin{equation*}
    \DTV{ \mathbb{P}_{\bm{X}} }{ \mathbb{P}_{\bm{Y}} }
    = 
    \mathbb{P} \left\{
      \frac{
        c - \sum_{i = 1}^n \mathbb{E} \{ Z_i \}
      }{
        \sqrt{\sum_{i = 1}^n \Var \{ Z_i \}}
      }
      \leq Z \leq
      \frac{c - \sum_{i = 1}^n \mathbb{E} \{ \widetilde{Z}_i \}}{
        \sqrt{\sum_{i = 1}^n \Var \{ \widetilde{Z}_i \}}
      }
    \right\}
    + \varepsilon,
  \end{equation*}
  with \(Z\sim N(0,1)\) and
  \(|\varepsilon|\leq C_0(\kappa_{n}+\widetilde{\kappa}_{n})\).

\end{proof}

\section{Appendix}%
\label{sec:appendix}
\subsection{Proof of Lemma~\ref{covariance_reduction}.}
Let \(\bm{X} \sim \mathcal{E} ( g_1, \Sigma_1 )\) and
\(\bm{Y} \sim \mathcal{E} ( g_2, \Sigma_2 )\). Since the matrix
\( \Sigma_1^{-1/2} \Sigma_2 \Sigma_1^{-1/2} \) is symmetric, there exists an
orthogonal matrix \( P \) and a diagonal matrix \( D \) such that
  \begin{equation*}
   \Sigma_1^{-1/2} \Sigma_2 \Sigma_1^{-1/2} = P^\top D P.
  \end{equation*}

  Changing the variables by \(\vk y = P \Sigma_1^{-1/2} \vk x\) we obtain
  \begin{align*}
    \Df [\mathcal{E} (g_1, \Sigma_1)][\mathcal{E} (g_2, \Sigma_2)]
    & = \int_{\R^n}
    \frac{1}{\sqrt{\det \Sigma_2}} \, 
    g_2 \left( \vk x^\top \Sigma_2^{-1} \vk x \right)
    f \left(
    \frac{
    g_1 \left( \vk x^\top \Sigma_{1}^{-1} \vk x \right)\sqrt{\det{\Sigma_2}}
    }{
    g_2 \left(\vk x^\top \Sigma_{2}^{-1}\vk x \right)\sqrt{\det{\Sigma_1}}
    }\right)
    \, d\vk x
   \\[7pt]
   & = \sqrt{\det\Sigma_1} \int_{\R^n}
   \frac{1}{\sqrt{\det \Sigma_2}} \, 
   g_2 \left( \vk y^\top D^{-1} \vk y \right)
   f \left(
    \frac{
    g_1 \left( \vk y^\top \vk y\right)\sqrt{\det D}
   }{
    g_2 \left(\vk y^\top D^{-1}\vk y \right)
    }\right)
    \, d\vk y
    \\[7pt]
    & = \int_{\R^n} 
    \frac{1}{\sqrt{ \det D }} \,
    g_2 \left( \vk y^\top D^{-1} \vk y \right)
    f \left(
    \frac{
    g_1 \left( \vk y^\top \vk y\right)\sqrt{\det D}
    }{
    g_2 \left(\vk y^\top D^{-1}\vk y \right)
    }\right)
    \, d\vk y
    \\ & 
    = \Df [\mathcal{E} (g_1, I)][\mathcal{E} (g_2, D)].
  \end{align*}
  To conclude the proof, it remains to notice that the eigenvalues and their
  multiplicities of \( \Sigma_1^{-1/2} \Sigma_2 \Sigma_1^{-1/2} \) and those of
  \( \Sigma_2 \Sigma_1^{-1} \) are the same.

\subsection{Proof of Proposition~\ref{tv_gen}}
    \begin{enumerate}[(a)]
      \item Follows from~\cite[Theorem 2.9]{FangKotz1990}.

      \item Observe that
      \begin{align*}\label{f_1>f_2}
        &
          \DTV{ \mathcal{E}(g_{1}, I) }{ \mathcal{E}(g_{2}, D) }
          = \Prob{ f_1(\vk X) \geq f_2(\vk X) }
          - \Prob{ f_1(\vk Y)\geq f_2(\vk Y) }
          \\[7pt]
        & \hspace{2cm} =
          \Prob{
          \left(\prod_{i=1}^{n}d_i\right)^{1/2}g_1(|\vk X|^2)
          \geq g_2\left(\sum_{i=1}^{n}\frac{X_{i}^2}{d_i}\right)
          }
          - \Prob{
          \left(\prod_{i=1}^{n}d_i\right)^{1/2}g_1(|\vk Y|^2)
          \geq g_2\left(\sum_{i=1}^{n}\frac{Y_{i}^2}{d_i}\right)
          }.
      \end{align*}
      Using that \( g_{1} \) and \( g_{2} \) are decreasing, we obtain
      \begin{equation*}
        \Prob{ |\bm{X}|^{2} \in A_{-} }
        \leq
        \Prob{
          \left(\prod_{i=1}^{n}d_i\right)^{1/2}g_1(|\vk X|^2)
          \geq g_2\left(\sum_{i=1}^{n}\frac{X_{i}^2}{d_i}\right)
        }
        \leq \mathbb{P} \{ |\bm{X}|^{2} \in A_{+} \}.
      \end{equation*}
      Similarly,
      \begin{equation*}
        \Prob{ d_{+} \sum_{i = 1}^n \frac{Y_{i}^{2}}{d_{i}} \in A_{-} }
        \leq
        \Prob{
          \left(\prod_{i=1}^{n}d_i\right)^{1/2}g_1(|\vk Y|^2)
          \geq g_2\left(\sum_{i=1}^{n}\frac{Y_{i}^2}{d_i}\right)
        }
        \leq
        \Prob{ d_{-} \sum_{i = 1}^n \frac{Y_{i}^{2}}{d_{i}} \in A_{+} }.
      \end{equation*}
      Combining these bounds, we obtain
      \begin{align*}
        &
          \Prob{ |\bm{X}|^{2} \in A_{-} }
          - \Prob{ d_{-} \sum_{i = 1}^n \frac{Y_{i}^{2}}{d_{i}} \in A_{+} }
        \\[7pt]
        & \hspace{2cm}
          \leq
          \DTV{ \mathcal{E} (g_{1}, I) }{ \mathcal{E} (g_{2}, D) }
          \leq
          \Prob{ |\bm{X}|^{2} \in A_{+} }
          - \Prob{ d_{+} \sum_{i = 1}^n \frac{Y_{i}^{2}}{d_{i}} \in A_{-} }.
      \end{align*}
      It remains to note that
      \begin{equation*}
        \Prob{ |\bm{X}|^{2} \in A_{\pm} }
        = \int_{\R^{n}}
        \Ind{ |\bm{x}|^{2} \in A_{\pm} } \, g_{1} ( \bm{x} )
        \, d\bm{x}
        = \frac{2 \pi^{n / 2}}{\Gamma ( n / 2 )}
        \int_0^{\infty}
        \Ind{ r^{2} \in A_{\pm} } \, g_{1} ( r^{2} )
        \, r^{n-1} \, dr
      \end{equation*}
      and
      \begin{align*}
        \Prob{ d_{\mp} \sum_{i = 1}^n \frac{Y_{i}^{2}}{d_{i}} \in A_{\pm} }
        & = \int_{\R^{n}}
        \Ind{ d_{\mp} \sum_{i = 1}^n \frac{y_{i}^{2}}{d_{i}} \in A_{\pm} }
        \, \prod_{i = 1}^n d_{i}^{-1/2}
          \, g_{2} \left( \sum_{i = 1}^n \frac{y_{i}^{2}}{d_{i}} \right)
          \, d\bm{y}
        \\[7pt]
        & = \int_{\R^{n}}
          \Ind{ d_{\mp} | \bm{y} |^{2} \in A_{\pm} } \, g_{2} ( | \bm{y} |^{2} )
          \, d\bm{y}
        \\[7pt]
        & = \frac{2 \pi^{n / 2}}{\Gamma ( n / 2 )}
          \int_0^{\infty}
          \Ind{ d_{\mp} r^{2} \in A_{\pm} } \, g_{2} ( r^{2} ) \, dr.
      \end{align*}

    \end{enumerate}
    
    \subsection{Proof of Proposition~\ref{kl_gen}}
    \begin{enumerate}[(a)]
      \item Follows from~\cite[Theorem 2.9]{FangKotz1990}.

      \item By definition, if \(g_2\) is a decreasing function, then we have that
      \begin{align*}
        \DKL{ \mathcal{E}(g_1,I) }{ \mathcal{E}(g_2,D) }
        &
          = \int_{\mathbb{R}^n} g_1(|\vk x|^2)
          \ln\left(
          \frac{
          g_1(|\vk x|^2)\prod_{i=1}^{n}d_i^{1/2}
          }{
          g_2\left(\sum_{i=1}^{n}\frac{x_i^2}{d_i}\right)
          }
          \right)
          d\vk x
        \\
        &
          = \frac{1}{2} \sum_{i=1}^{n}\ln d_i
          + \int_{\mathbb{R}^n}
          g_1(|\vk x|^2)
          \ln\left(
          \frac{g_1(|\vk x|^2)}{g_2\left(\sum_{i=1}^{n}\frac{x_i^2}{d_i}\right)}
          \right)
          d\vk x
        \\
        & \leq
          \frac{1}{2} \sum_{i=1}^{n}\ln d_i
          + \int_{\mathbb{R}^n}
          g_1(|\vk x|^2)
          \ln\left(
          \frac{g_1(|\vk x|^2)}{g_2\left(|\vk x|^2/d_-\right)}
          \right)
          d\vk x
        \\
        &
          = \frac{1}{2}\sum_{i=1}^{n} \ln d_i+\frac{2\pi^{n/2}}{\Gamma(n/2)}
          \int_{0}^{\infty}
          g_1(r^2)
          \ln\left(\frac{g_1(r^2)}{g_2\left(r^2/d_-\right)}\right)
          \, r^{n-1} \, dr.
      \end{align*}
      The lower bound for
      \( \DKL{ \mathcal{E}(g_1,I) }{ \mathcal{E}(g_2,D) }\) may be proved in
      the same way.
      \smallskip

      \subsection{Proof of Theorem~\ref{main_TVD}}
      To start with, we check auxiliary lemmas mentioned in
      \Cref{sec:main_TV}.
      \begin{proof}[Proof of \Cref{sublevel-lemma}]
        If \( \alpha > \gamma \), then the maximum of
        \( \varphi_{\alpha, \gamma} ( x ) \), \( x \geq 0 \) equals
        \begin{equation*}
          \varphi_{\alpha, \gamma} ( x_0 )
          = e^{\gamma} \left( \frac{\alpha}{\gamma e} \right)^{\alpha} > 1
          \quad \text{and is attained at} \quad
          x_0 = \frac{\alpha - \gamma}{\gamma}.
        \end{equation*}
        Otherwise, if \( \gamma \geq \alpha \), the maximum of
        \( \varphi_{\alpha, \gamma} ( x ) \), \( x \geq 0 \) is
        \( \varphi_{\alpha, \gamma} ( 0 ) = 1 \) and the function
        \( \varphi_{\alpha, \gamma} \) is decreasing. Therefore, if
        \( \alpha < \gamma \) and \( c \geq 1 \) or \( \gamma \leq \alpha \) and
        \( c \geq e^{\gamma} ( \alpha / \gamma e )^{\alpha} \), we have that
        \begin{equation*}
          D_{\alpha, \gamma} ( c ) = \{ x \colon \varphi_{\alpha, \gamma} ( x ) \leq c \} = [0, \infty).
        \end{equation*}
        In order to find the roots of \( \varphi_{\alpha, \gamma} ( x ) = c \)
        in the remaining cases
        \begin{itemize}
          \item[(a)] \( 0 < c < 1 \),
          \item[(b)] \( \alpha \geq \gamma \) and
          \( 1 < c < e^{\gamma} ( \alpha / \gamma e )^{\alpha} \),
        \end{itemize}
        we shall use the Lambert functions \( W_{-1} ( x ) \) and
        \( W_0 ( x ) \) defined on \( [-1/e, 0] \) as the two solutions
        \( y_{-1} \) and \( y_0 \) of \( y e^y = x \) satisfying
        \( y_{-1} \leq y_0 \), see~\cite{CorlessEtAl1996}. Let us first rewrite
        the equation in the form \( y e^y = z \) with some \( y \) and \( z \):
        \begin{align}
          ( 1 + x )^{\alpha} e^{-\gamma x}
          = c
          & \iff
            ( 1 + x )^{\alpha} e^{-\gamma ( 1 + x )} = c e^{-\gamma}
            \notag
          \\[7pt]
          & \iff
            ( 1 + x ) e^{-( \gamma / \alpha ) ( 1 + x )} = c^{1/\alpha} e^{-\gamma / \alpha}
            \notag
          \\[7pt]
          & \iff
            -\frac{\gamma}{\alpha} \, ( 1 + x ) \, e^{-( \gamma / \alpha ) ( 1 + x ) }
            = -\frac{ \gamma c^{1/\alpha} e^{-\gamma / \alpha} }{\alpha}.
            \label{equation1}
        \end{align}
        The last equation is of the form \( y e^y = z \) with
        \( y = -( \gamma / \alpha ) ( 1 + x ) \) and
        \( z = -\gamma c^{1/\alpha} e^{-\gamma / \alpha} / \alpha \). In order
        to solve it, we need to check whether \( z \in [-1/e, 0] \), which is
        the domain of \( W_{-1} \) and \( W_0 \):
        \begin{equation}
          \label{condition1}
          -\frac{ \gamma c^{1/\alpha} e^{-\gamma / \alpha} }{\alpha} \geq -\frac{1}{e}
          \iff
          c^{1/\alpha} \leq \frac{\alpha}{\gamma} \, e^{\gamma / \alpha - 1}
          \iff
          c \leq e^{\gamma} \left( \frac{\alpha}{\gamma e} \right)^{\alpha}.
        \end{equation}
        We have:
        \begin{itemize}
          \item[(a)] If \( 0 < c < 1 \), then \( c^{1/\alpha} < 1 \) and
                \( e^{x-1} / x \geq 1 \) for all \( x \geq 0 \),
                so~\eqref{condition1} is satisfied.
          \item[(b)] If \( \alpha \geq \gamma \) and
                \( 1 < c < e^{\gamma} ( \alpha / \gamma e )^{\alpha} \), then
                the condition~\eqref{condition1} is also satisfied.
        \end{itemize}
        Applying \( W_0 \) and \( W_{-1} \) to both sides of~\eqref{equation1},
        we obtain two solutions which we denote by
        \( a_{\alpha, \gamma} ( c ) \) and \( b_{\alpha, \gamma} ( c ) \):
        \begin{equation*}
          -\frac{\gamma}{\alpha} \, ( 1 + a_{\alpha, \gamma} ( c ) )
          = W_0 \left( -\frac{\gamma c^{1/\alpha} e^{-\gamma / \alpha} }{\alpha} \right)
          \quad \text{and} \quad
          -\frac{\gamma}{\alpha} \, ( 1 + b_{\alpha, \gamma} ( c ) )
          = W_{-1} \left( -\frac{\gamma c^{1/\alpha} e^{-\gamma / \alpha} }{\alpha} \right).
        \end{equation*}
        Solving them for \( a_{\alpha, \gamma} ( c ) \) and \( b_{\alpha, \gamma} ( c ) \), we obtain
        \begin{equation*}
          a_{\alpha, \gamma} ( c )
          = -1 -\frac{\alpha}{\gamma} \,
          W_0 \left( -\frac{\gamma c^{1/\alpha} e^{-\gamma / \alpha} }{\alpha} \right)
          \quad \text{and} \quad
          b_{\alpha, \gamma} ( c )
          = -1 -\frac{\alpha}{\gamma} \,
          W_{-1} \left( -\frac{\gamma c^{1/\alpha} e^{-\gamma / \alpha} }{\alpha} \right).
        \end{equation*}
        Since \( W_{-1} ( x ) \leq W_0 ( x ) \), we see that
        \( a_{\alpha, \gamma} ( c ) \leq b_{\alpha, \gamma} ( c ) \). If
        \( \gamma > \alpha \), using that \( W_0 ( z ) \geq -1 \)
        (see~\cite[p.~331]{CorlessEtAl1996}) we obtain
        \begin{equation*}
          a_{\alpha, \gamma} ( c ) \leq -1 + \frac{\alpha}{\gamma} < 0,
        \end{equation*}
        hence \( b_{\alpha, \gamma} ( c ) \) is the only root in \( x \geq 0 \).
        If, on the other hand, \( \gamma \leq \alpha \) and \( 0 < c < 1 \),
        then by \( W_0 ( z ) \geq \sqrt{1 + ez} - 1 \) for \( x \in [-1/e, 0] \)
        (see~\cite[Theorem 3.2]{Roigsolvas2022}), it follows that
        \begin{align*}
          a_{\alpha, \gamma} ( c )
          & \leq -1 - \frac{\alpha}{\gamma}
            \left( \sqrt{1 - \frac{\gamma}{\alpha} \, c^{1/\alpha} \, e^{1-\gamma / \alpha} } - 1 \right)
          \\[7pt]
          & < -1 - \frac{\alpha}{\gamma}
            \left( \sqrt{1 - \frac{\gamma}{\alpha} \, e^{1-\gamma/\alpha} } - 1 \right)
            \quad \text{by} \quad 0 < c < 1
          \\[7pt]
          & \leq -1 - \frac{\alpha}{\gamma}
            \left( 1 - \frac{\gamma}{\alpha} - 1 \right)
            \quad \text{by} \quad \sqrt{ 1 - x e^{1-x} } \leq 1 - x, \ x \in [0, 1]
          \\[7pt]
          & = 0.
        \end{align*}
        Hence, in this case \( b_{\alpha, \gamma} ( c ) \) is again the only
        root in \( x \geq 0 \).

        Finally, if \(\alpha\geq\gamma\) and
        \( 1 < c < e^{\gamma} ( \alpha / \gamma e )^{\alpha} \), then since
        \(W_0(x)\geq -1\), we have that
        \begin{align*}
          a_{\alpha, \gamma} ( c )
          \geq -1 + \frac{\alpha}{\gamma}\geq 0.
        \end{align*}
      \end{proof}
      From the presented proof \Cref{Lambert} automatically follows.

      \begin{proof}[Proof of \Cref{ndown}]
        The function \( n \mapsto \clow(n) \) increases on \( \N \) if for all
        \( k \in \N \) holds
        \begin{equation}
          \frac{\clow(2k+1)}{\clow(2k)} > 1
          \quad \text{and} \quad
          \frac{\clow(2k)}{\clow(2k-1)} > 1.
        \end{equation}
        Since
        \begin{equation*}
          \frac{
            \Gamma \left( \frac{\nu +1}{2}  + k \right)
          }{
            \Gamma \left( \frac{\nu}{2} + k \right)
          }
          = \frac{
            \Gamma \left( \frac{\nu +1}{2} \right)
          }{
            \Gamma \left( \frac{\nu}{2} \right)
          }
          \prod_{l = 0}^{k-1}
          \frac{( \nu +1) / 2 + l}{\nu/ 2 + l}
          = \frac{
            \Gamma \left( \frac{\nu + 1}{2} \right)
          }{
            \Gamma \left( \frac{\nu}{2} \right)
          }
          \prod_{l = 0}^{k-1} \left( 1 + \frac{1}{\nu+ 2l} \right),
        \end{equation*}
        we have
        \begin{align*}
          \frac{
          \clow(2k+1)}{\clow(2k)}
          =
          \underbrace{
          \frac{
          \Gamma \left( \frac{\nu + 1}{2} \right)
          }{
          \Gamma \left( \frac{\nu}{2} \right)
          }
          \left(\frac{2}{\nu}\right)^{1/2}d_{-}^{1/2}
          }_{= \, \clow (1) \, \geq \, 1}
          \prod_{l = 0}^{k-1} \left( 1 + \frac{1}{\nu + 2l} \right)
          \geq \prod_{l = 0}^{k-1} \left( 1 + \frac{1}{\nu + 2l} \right)
          > 1.
        \end{align*}
        Moreover
        \begin{equation*}
          \frac{
            \Gamma \left( \frac{\nu}{2} + k  \right)
          }{
            \Gamma \left( \frac{\nu + 1}{2} + k-1 \right)
          }
          = \frac{
            \Gamma \left( \frac{\nu}{2} + 1 \right)
          }{
            \Gamma \left( \frac{\nu + 1}{2} \right)
          }
          \prod_{l = 0}^{k-2}
          \frac{\nu / 2 + 1 + l}{(\nu + 1) / 2 + l}
          = \frac{
            \Gamma \left( \frac{\nu}{2} + 1 \right)
          }{
            \Gamma \left( \frac{\nu + 1}{2} \right)
          }
          \prod_{l = 0}^{k-2}
          \left( 1 + \frac{1}{\nu + 2l + 1} \right).
        \end{equation*}
        Using the logarithmic convexity of the Gamma-function, we have that
        \(\Gamma(x) \, \Gamma(x+1)\geq\Gamma(x+1/2)^2\) for all \( x \geq 0 \),
        and therefore
        \begin{align*}
          \frac{\clow(2k)}{\clow(2k-1)}
          &
            = \frac{
            \Gamma \left( \frac{\nu }{2}+1 \right)
            }{
            \Gamma \left( \frac{\nu+1}{2} \right)
            }
            \left(\frac{2}{\nu}\right)^{1/2}d_{-}^{1/2}
            \prod_{l = 0}^{k-2} \left( 1 + \frac{1}{\nu + 2l + 1} \right)
          \\[7pt]
          &
            \geq
            \underbrace{
            \frac{
            \Gamma \left( \frac{\nu +1}{2}\right)
            }{
            \Gamma \left( \frac{\nu}{2} \right)
            }
            \left(\frac{2}{\nu}\right)^{1/2}d_{-}^{1/2}
            }_{= \, \clow ( 1 ) \, \geq \, 1}
            \prod_{l = 0}^{k-2} \left( 1 + \frac{1}{\nu + 2l + 1} \right)
            \geq \prod_{l = 0}^{k-2} \left( 1 + \frac{1}{\nu + 2l+1} \right)
            > 1.
        \end{align*}
      \end{proof}
      \begin{proof}[Proof of \Cref{minimum}]
        The derivative \( \clow' ( x ) \) is given by
        \[
          \clow'(x)
          =f(x)
          \left(
            \ln\left(\frac{2d_{-}}{\nu}\right)
            +\psi_0\left(\frac{\nu+x}{2}\right)
          \right),
        \]
        where
        \[
          f(x)
          =\frac{
            2^{-1+x/2}d_{-}^{x/2}\Gamma((\nu+x)/2)
          }{
            \nu^{x/2}\Gamma(\nu/2)
          }>0
          \quad \text{for all} \quad x > 0
        \]
        and \(\psi_0(x)\) is a digamma function. Since \(\psi_0\) is strictly
        increasing, \( \clow \) has at most one minimum point \(x_{\min}\). Let
        us check that if \(d_{-}<1\), then \( \clow \) has minimum
        \(x_{\min}>0\). Indeed, using that
        \[
          \psi_0(x)<\ln x \ \text{ for } \ x>0,
        \]
        we obtain
        \[
          c_{\min}'(0)
          =f(0) \left(
            \ln\left(\frac{2d_{-}}{\nu}\right)
            +\psi_0\left(\frac{\nu}{2}\right)
          \right)<f(0)\ln d_{-}<0.
        \]
        Hence, the minimum \(x_{\min}\) exists, and \( \clow \) increases if
        \(x>x_{\min}\) and decreases otherwise for \(x>0\). Moreover,
        \[
          0 = \clow'(x_{\min})
          <f(x_{\min}) \left(
            \ln\left(\frac{2d_{-}}{\nu}\right)
            +\ln\left(\frac{\nu+x_{\min}}{2}\right)
          \right)
          =f(x_{\min}) \ln\left(
            \frac{2d_{-}(\nu+x_{\min})}{2\nu}
          \right),
        \]
        hence \( x_{\min} > \nu / d_{-} - \nu \). Here we used again that
        \(\psi_0(x)<\ln x\) for \(x>0\).
      \end{proof}

      \begin{proof}[Proof of \Cref{n1}]
        Observe first that
        \(c/(e^{\gamma_{-} } ( \alpha / \gamma_{-} e )^{\alpha})\) decreases as
        on \(n\in\mathbb{N}\). To this end, notice that
        \[
          e^{\gamma_{-} } ( \alpha / \gamma_{-}  e )^{\alpha}
          =C_1e^{-n/2} ( \nu + n )^{( \nu + n ) / 2} \nu^{-n/2}
          d_{+}^{n / 2},
        \]
        where \( C_1 = e^{\nu/2d_{+} - \nu/2} d_{+}^{\nu/2} \nu^{-\nu/2}\).
        Therefore, we obtain that
        \begin{align*}
          \frac{c}{e^{\gamma_{-} } ( \alpha / \gamma_{-}  e )^{\alpha}}
          & = \frac{1}{C_1}
            \frac{\Gamma ( ( \nu + n ) / 2 )(2e)^{n/2}}{ \Gamma ( \nu / 2 )(\nu+n)^{(\nu+n)/2} }
            \prod_{i=1}^{n}\left(\frac{d_{i}}{d_{+}}\right)^{1/2}
            =C_2f(n)g(n),
        \end{align*}
        where \(C_2=1/(C_1\Gamma(\nu/2)) \) is an independent of \(n\) positive
        constant,
        \[
          f(n)
          =\frac{\Gamma ( ( \nu + n ) / 2) (2e)^{n/2}}{(\nu+n)^{(\nu+n)/2}}
          \quad \text{and} \quad
          g(n)
          =\prod_{i=1}^{n}\left(\frac{d_{i}}{d_{+}}\right)^{1/2}.
        \]
        Since \(f,g\) are positive functions and \(g\) is decreasing, it remains
        to show that \(f\) decreases on \(\N\). To this end, we calculate its
        derivative. It turns out that
        \[
          f'(n)
          = f_0(n) \left(
            \ln(2)-\ln(\nu+n)
            +\psi_0\left(\frac{\nu+n}{2}\right)
          \right)
          = f_0(n) \left(
            -\ln\left(\frac{\nu+n}{2}\right)
            +\psi_0\left(\frac{\nu+n}{2}\right)
          \right)
          < 0,
        \]
        where
        \[
          f_0(n)=\frac{(2e)^{n/2}\Gamma(( \nu+n)/2)}{2(\nu+n)^{(\nu+n)/2}}>0.
        \]
        Hence, \(f\) is a decreasing function and so does
        \(c/(e^{\gamma_{-} } ( \alpha / \gamma_{-} e )^{\alpha})\). To finish
        the prove it remains to check the inequality for \(n=1.\) Since
        \(d_1\leq d_{+}\), we have that
        \[
          \frac{c}{e^{\gamma_{-} } ( \alpha / \gamma_{-}  e )^{\alpha}
          }
          =\frac{1}{C_1}\frac{
            \Gamma ( ( \nu + 1 ) / 2 )(2e)^{1/2}
          }{ \Gamma ( \nu / 2 )(\nu+1)^{(\nu+1)/2} }\left(\frac{d_{1}}{d_{+}}\right)^{1/2}
          \leq \frac{1}{C_1}\frac{
            \Gamma ( ( \nu + 1) / 2 )(2e)^{1/2}
          }{ \Gamma ( \nu / 2 )(\nu+1)^{(\nu+1)/2} },
        \]
        where \( C_1 = e^{\nu/2d_{+} - \nu/2} d_{+}^{\nu/2} \nu^{-\nu/2}.\)
        Since \(xe^{1/x}\geq e\) for all \(x>0\), we have that
        \(C_1\geq \nu^{-\nu/2}\). Therefore, using that
        \(\Gamma(x+1/2)\leq x^{1/2}\Gamma(x)\) for all \(x>0\), we obtain that
        \begin{align*}
          \frac{c}{e^{\gamma_{-} } ( \alpha / \gamma_{-}  e )^{\alpha}
          }
          &\leq\frac{
            \nu^{\nu/2}\Gamma ( ( \nu + 1 ) / 2 )(2e)^{1/2}
            }{ \Gamma ( \nu / 2 )(\nu+1)^{(\nu+1)/2} }
            \leq\frac{
            e^{1/2}
            }{ (1+1/\nu)^{(\nu+1)/2} }
            =\left(\frac{e}{(1+1/\nu)^{\nu+1}}\right)^{1/2}<1,
        \end{align*}
        where we used that \((1+1/x)^{x+1}>e\) for all \(x>0\).
      \end{proof}

      \begin{proof}[Proof of Theorem~\ref{main_TVD}]
        If~\eqref{limsup-condition} is satisfied,
        \( D_{\alpha, \gamma_+} ( c ) = [0, \infty) \) and therefore
        \( A_+ = [0,\infty)\). This makes the lower bound from
        Theorem~\ref{tv_gen} trivial, since
        \[
          \frac{1}{\Gamma(n/2)}\int_{0}^{\infty}y^{n/2-1}e^{-y}dy=1.
        \]
        Nevertheless, the upper bound remains interesting:
        \begin{align*}
          \DTV{ \mathcal{E}(g_1,I) }{ \mathcal{E}(g_2,D) }
          &\leq \frac{2\pi^{n/2}}{\Gamma(n/2)}
            \int_{0}^{\infty}
            g_1\left(r^2\right)
            \, r^{n-1} \, dr
            -\frac{1}{\Gamma(n/2)}
            \int_{0}^{\infty}
            \Ind{ 2d_{+}y\in A_{-} }
            y^{n/2-1}e^{-y}
            \, dy
          \\
          &
            = \frac{1}{\Gamma(n/2)}
            \int_{0}^{\infty}
            \Ind{ 2d_{+}y\in A^c_{-} }
            y^{n/2-1}e^{-y}
            \, dy,
        \end{align*}
        where \(A_{-}^c=[0,\infty)\setminus A_{-}\). Since \( A^c_{-} =\left(\nu
        a_{\alpha, \gamma_-} ( c ), \nu b_{\alpha, \gamma_-} ( c ) \right) \),
        we have that
        \begin{align*}
          \DTV{ \mathcal{E}(g_1,I) }{ \mathcal{E}(g_2,D) }
          & \leq
          \frac{1}{\Gamma(n/2)}
          \int_{
          \nu a_{\alpha, \gamma_-} ( c )/2d_{+}
          }^{
          \nu b_{\alpha, \gamma_-} ( c )/2d_{+}
          }
          y^{n/2-1}e^{-y}dy
          \\[7pt]
          & = \frac{
          \gamma(n/2, \nu b_{\alpha, \gamma_-} ( c )/(2d_{+}))
          -\gamma(n/2, \nu a_{\alpha, \gamma_-} ( c )/(2d_{+}))
          }{\Gamma(n/2)}
          \\[7pt]
          & = 
          P \left( \frac{\nu b_{\alpha, \gamma_- ( c )}}{2 d_+} ; \frac{n}{2} \right) 
          - P \left( \frac{\nu a_{\alpha, \gamma_- ( c )}}{2 d_+} ; \frac{n}{2} \right),
        \end{align*}
        where \( P ( z ; a ) = \gamma ( z ; a ) / \Gamma ( a ) \) is the lower
        regularized incomplete Gamma function. Assume now that the
        condition~\eqref{liminf-condition} holds. Then we have that
        \begin{equation*}
          0 < a_{\alpha, \gamma_{\pm}} ( c )
          < b_{\alpha, \gamma_{\pm}} ( c )
          < \infty.
        \end{equation*}
        Since
        \(A_{\pm}^c=(\nu a_{\alpha, \gamma_{\pm}} ( c ), \nu b_{\alpha, \gamma_{\pm}} ( c ))\),
        by Proposition~\ref{tv_gen} we have the following bounds on
        \( d(\mathcal{E}(g_1,I), \mathcal{E}(g_2,D))\):
        \begin{align*}
          &
          \DTV{ \mathcal{E}(g_1,I) }{ \mathcal{E}(g_2,D) }
          \\[7pt]
          & \hspace{1cm}
          \leq \frac{2\pi^{n/2}}{\Gamma(n/2)}
          \int_{0}^{\infty}
          \Ind{ r^2\in A_{+}}
          g_1(r^2)
          \, r^{n-1} \, dr
          -\frac{1}{\Gamma(n/2)}
          \int_{0}^{\infty}
          \Ind{ 2d_{+}y\in A_{-} }
          y^{n/2-1}e^{-y}
          \, dy
        \end{align*}
        and
        \begin{align*}
          &
          \DTV{ \mathcal{E}(g_1,I) }{ \mathcal{E}(g_2,D) }
          \\[7pt]
          & \hspace{1cm}
            \geq
            \frac{2\pi^{n/2}}{\Gamma(n/2)}
            \int_{0}^{\infty}
            \Ind{ r^2\in A_{-}}
            g_1(r^2)
            \, r^{n-1} \, dr
            -\frac{1}{\Gamma(n/2)}
            \int_{0}^{\infty}
            \Ind{ 2d_{-}y\in A_{+} }
            y^{n/2-1}e^{-y}
            \, dy.
        \end{align*}
        Since \(f_1\) is a pdf of \(\vk X\), we can rewrite the bounds in terms
        of \(A_{\pm}^c\) as follows:
        \begin{align*}
          &
          \DTV{ \mathcal{E}(g_1,I) }{ \mathcal{E}(g_2,D) }
          \\[7pt]
          & \hspace{1cm}
            \leq \frac{2\pi^{n/2}}{\Gamma(n/2)}
            \int_{0}^{\infty}\Ind{r^2\in A_{+}}
            g_1\left(r^2\right)
            r^{n-1}dr
            -\frac{1}{\Gamma(n/2)}
            \int_{0}^{\infty}
            \Ind{ 2d_{+}y\in A_{-} }
            y^{n/2-1}e^{-y}
            \, dy
          \\[7pt]
          & \hspace{1cm}
            = \frac{1}{\Gamma(n/2)}
            \int_{0}^{\infty}
            \Ind{2d_{+}y\in A^c_{-}}
            y^{n/2-1}e^{-y}dy
            -\frac{2\pi^{n/2}}{\Gamma(n/2)}
            \int_{0}^{\infty}
            \Ind{ r^2\in A^c_{+}}
            g_1(r^2)
            \, r^{n-1} \, dr
          \\[7pt]
          & \hspace{1cm}
            = P \left( \frac{\nu b_{\alpha, \gamma_-}}{2 d_+} ; \frac{n}{2} \right) 
            - P \left( \frac{\nu a_{\alpha, \gamma_-}}{2 d_+} ; \frac{n}{2} \right) 
            -\frac{2\pi^{n/2}}{\Gamma(n/2)}
            \int_{
            \nu a_{\alpha, \gamma_{+}} ( c )
            }^{
            \nu b_{\alpha, \gamma_{+}} ( c )
            }
            g_1(r^2)
            \, r^{n-1} \, dr
        \end{align*}
        Similarly, we obtain that
        \begin{align*}
          &
            \DTV{ \mathcal{E}(g_1,I) }{ \mathcal{E}(g_2,D) }
          \\[7pt]
          & \hspace{1cm}
            \geq 
            P \left( \frac{\nu b_{\alpha, \gamma_+}}{2 d_-} ; \frac{n}{2} \right) 
            - P \left( \frac{\nu a_{\alpha, \gamma_+}}{2 d_-} ; \frac{n}{2} \right) 
            -\frac{2\pi^{n/2}}{\Gamma(n/2)}
            \int_{
            \nu a_{\alpha, \gamma_{-}} ( c )
            }^{
            \nu b_{\alpha, \gamma_{-}} ( c )
            }
            g_1(r^2)
            \, r^{n-1} \, dr.
        \end{align*}
        It remains to calculate the integral of \( g_1 \):
        \begin{align*}
          &
            \frac{2\pi^{n/2}}{\Gamma(n/2)}
            \int_{
            \nu a_{\alpha, \gamma_{\pm}} ( c )
            }^{
            \nu b_{\alpha, \gamma_{\pm}} ( c )
            }
            g_1(r^2)
            \, r^{n-1} \, dr
          \\[7pt]
          & \hspace{1cm}
            = \frac{2\Gamma ( ( \nu + n ) / 2 )}{\Gamma(n/2)\Gamma ( \nu / 2 ) \nu^{n/2}}
            \int_{
            \nu a_{\alpha, \gamma_{\pm}} ( c )
            }^{
            \nu b_{\alpha, \gamma_{\pm}} ( c )
            }
            \left[ 1 + \frac{r^2}{\nu}\right]^{-( \nu + n ) / 2}
            \, r^{n-1} \, dr
          \\[7pt]
          & \hspace{1cm}
            = \frac{\Gamma ( ( \nu + n ) / 2 )}{\Gamma(n/2)\Gamma ( \nu / 2 )}
            \int_{
            \nu a^2_{\alpha, \gamma_{\pm}} ( c )
            }^{
            \nu b^2_{\alpha, \gamma_{\pm}} ( c )
            }
            \frac{t^{n / 2 - 1}}{(1 + t)^{n / 2 + \nu / 2} }
            \, dt
          \\[7pt]
          & \hspace{1cm}
            = \frac{\Gamma ( ( \nu + n ) / 2 )}{\Gamma(n/2)\Gamma ( \nu / 2 )}
            \left(
              B \left( \frac{1}{1 + \nu a^2_{\alpha, \gamma_{\pm }}} ; \frac{\nu}{2}, \frac{n}{2} \right)
              -B \left( \frac{1}{1 + \nu b^2_{\alpha, \gamma_{\pm }}} ; \frac{\nu}{2}, \frac{n}{2} \right)
            \right)
          \\[7pt]
          & \hspace{1cm}
            =
            I \left( \frac{1}{1 + \nu a^2_{\alpha, \gamma_{\pm }}} ; \frac{\nu}{2}, \frac{n}{2} \right)
            -I \left( \frac{1}{1 + \nu b^2_{\alpha, \gamma_{\pm }}} ; \frac{\nu}{2}, \frac{n}{2} \right)
        \end{align*}
        where we used the following representation of the incomplete Beta function:
        \begin{equation*}
          B ( x ; a, b ) 
          = \int_{\frac{1-x}{x}}^\infty 
          \frac{t^{b - 1}}{(1 + t)^{a + b}} \, dt,
        \end{equation*}
        as well as
        \begin{equation*}
          \frac{\Gamma ( a + b )}{\Gamma ( a ) \, \Gamma ( b )} 
          \, B ( x ; a, b )
          = 
          \frac{B ( x; a, b )}{B ( a, b )}
          =
          I ( x; a, b ),
        \end{equation*}
        where \( I \) is the regularized incomplete Beta function.

      \end{proof}

      \subsection{Proof of Theorem~\ref{main_KL}}
      Let \(\vk X\sim \mathcal{E}(g_1,I)\) and
      \(\vk Y\sim \mathcal{E}(g_2,D)\).
      %
      By definition of \(\DKL\) we obtain that
      \begin{align*}
        &
          \DKL{ \mathcal{E}(g_1,I) }{ \mathcal{E}(g_2,D) }
        \\[7pt]
        & \hspace{1cm}
          =
          \frac{1}{2}\sum_{i=1}^{n}\ln d_i+\frac{n}{2}\ln(2\pi)
          +\frac{\nu}{\nu-2}\sum_{i=1}^{n}\frac{1}{2d_{i}}
          +\frac{2\pi^{n/2}}{\Gamma(n/2)}
          \int_{0}^{\infty}
          g_1(r^2)\ln\left(g_1(r^2)\right)
          \, r^{n-1} \, dr,
      \end{align*}
      where we used that \(\mathbb{E}\{X_{1}^2\}=\nu/(\nu-2)\) and \(g_2(t) = 1/(2\pi)^{n/2}e^{-t/2}\).
      %
      %
      The last integral on the right evaluates to
      \begin{equation*}
        -\frac{1}{2} (\nu+n) \left(
          \psi_{0}\left(\frac{\nu+n}{2}\right)
          -\psi_{0}\left(\frac{\nu}{2}\right)
        \right)
        + \ln\left(
          \frac{\Gamma ( ( \nu + n ) / 2 )}{\Gamma ( \nu / 2 ) \, \pi^{n / 2}\nu^{n/2}}
        \right).
      \end{equation*}
      Next, we estimate \(\DKL{ \mathcal{E}(g_2,D) }{ \mathcal{E}(g_1,I) }\). By
      \Cref{kl_gen} we have that
      \begin{align*}
        &
          \DKL{ \mathcal{E}(g_2,D) }{ \mathcal{E}(g_1,I) }
        \\[7pt]
        & \hspace{1cm}
          \geq
          \max\left(
          0,
          -\frac{1}{2^{n/2-1}\Gamma(n/2)}
          \int_{0}^{\infty}
          e^{-r^2/2}
          \ln(g_1(d_{-}r^2))
          \, r^{n-1} \, dr
          -\frac{n}{2}
          -\frac{n}{2}\ln(2\pi)
          -\frac{1}{2}\sum_{i=1}^{n}\ln d_i
          \right)
      \end{align*}
      and
      \begin{align*}
        &
          \DKL{ \mathcal{E}(g_2,D) }{ \mathcal{E}(g_1,I) }
        \\[7pt]
        & \hspace{1cm}
          \leq
          -\frac{1}{2^{n/2-1}\Gamma(n/2)}
          \int_{0}^{\infty}
          e^{-r^2/2}
          \ln(g_1(d_{+}r^2))
          \, r^{n-1} \, dr
          -\frac{n}{2}
          -\frac{n}{2}\ln(2\pi)
          -\frac{1}{2}\sum_{i=1}^{n}\ln d_i.
      \end{align*}

      It remains to calculate the first term on the right.
      %
      %
      We have
      \begin{align*}
        &
        \frac{1}{2^{n/2-1} \Gamma(n/2)}
        \int_{0}^{\infty}
        e^{-r^2/2}
        \ln(g_1(d_{\mp}r^2))
        \, r^{n-1} \, dr
        \\[7pt]
        & \hspace{1cm}
        = 
        -\frac{\nu+n}{2^{n/2}\Gamma(n/2)}
        \int_{0}^{\infty}
        e^{-r^2/2}
        \ln\left(
          1+\frac{d_{\mp}r^2}{\nu}
          \right)
        \, r^{n-1} \, dr
        + \ln\left(
          \frac{
            \Gamma ( ( \nu + n ) / 2 )
            }{
            \Gamma ( \nu / 2 ) \, \pi^{n / 2}\nu^{n/2}
            }
          \right)
        \\[7pt]
        & \hspace{1cm}
        = 
        -\frac{\nu+n}{2 \Gamma(n/2)}
        \int_{0}^{\infty}
        e^{-y}
        \ln\left(
          1+\frac{2 d_{\mp} y}{\nu}
          \right)
        \, y^{n / 2-1} \, dy
        + \ln\left(
          \frac{
            \Gamma ( ( \nu + n ) / 2 )
            }{
            \Gamma ( \nu / 2 ) \, \pi^{n / 2}\nu^{n/2}
            }
          \right)
      \end{align*}
      Denote \( \widetilde{\Delta}_{\mp, n} = \int_0^\infty e^{-y} \ln \left( 1
      + \frac{2 d_\mp y}{\nu} \right) \, y^{n / 2 - 1} \, dy \). We claim that
      \( \widetilde{\Delta}_{\mp, n} = \Delta_{\mp, n} \).
      For \( n = 1 \) and \( n = 2 \), the identities
      \begin{equation*}
        \widetilde{\Delta}_{\mp , n}
       = 
        \begin{dcases}
        \sqrt{\pi} \left( 
          \pi \erfi \left( \sqrt{ \frac{\nu}{2 d_{\mp }} }  \right) 
          + \ln \left( \frac{d_{\mp}}{2 \nu} \right)
          - \gammaEM
        \right) 
        - \frac{\sqrt{ \pi } \nu}{d_{\mp}}
        {}_2 F_2 \left( 
        1, 1; \frac{3}{2}, 2; \frac{\nu}{2 d_\mp}
        \right),
        & \text{if } n = 1,
        \\
        - e^{\nu / 2 d_{\mp}} 
        \Ei \left( -\frac{\nu}{2 d_{\mp}}\right),
        & \text{if } n = 2
        \end{dcases}
      \end{equation*}
      may be checked directly. To prove the case for all \( n \in \N \), we use
      mathematical induction. Assume that \(n\geq 3\) and
      \(\widetilde{\Delta}_{\mp, k} = \Delta_{\mp, k}\) for all \(k \leq n -
      1\). Let us show that it holds for \(k = n\). Using integration by parts,
      we obtain
      \begin{align*}
        &
        \int_{0}^{\infty}
        e^{-y}
        \ln\left(
          1+\frac{2d_{\mp}y}{\nu}
          \right)
          y^{n/2-1} dy
        \\[7pt]
        & \hspace{1cm}
        =
        -e^{-y}
        \ln\left(
          1+\frac{2d_{\mp}y}{\nu}
          \right)
          y^{n/2-1}\ \Bigg|_{0}^{\infty}
        + \left(\frac{n}{2}-1\right)
        \int_{0}^{\infty}
        e^{-y}
        \ln\left(1+\frac{2d_{\mp}y}{\nu}\right)
        y^{(n-2)/2-1} dy
        \\[7pt]
        & \hspace{1cm} \phantom{=} \
        + \frac{2d_{\mp}}{\nu}
        \int_{0}^{\infty}
        e^{-y}
        \left(1+\frac{2d_{\mp}y}{\nu}\right)^{-1}
        y^{n /2-1} dy
        \\[7pt]
        & \hspace{1cm}
        \eqqcolon \left(\frac{n}{2}-1\right)
        \widetilde{\Delta}_{\mp, n-2}
        + \widetilde{\alpha}_{\mp,n-2}
        \\[7pt]
        & \hspace{1cm}
        = \left(\frac{n}{2}-1\right)
        \left(
          \left(\frac{n}{2}-2\right)
          \Delta_{\mp, n-4}
          + \alpha_{\mp,n-4}
        \right)
        +\alpha_{\mp,n-2}
        \\[7pt]
        & \hspace{1cm} 
        = 
        \left(
          \prod_{i=1}^{k}\left(\frac{n}{2}-i\right)
        \right)
        \Delta_{\mp, n-2k}
        + \sum_{l=1}^{k} \left(
          \prod_{i=1}^{l-1}\left(\frac{n}{2}-i\right)
        \right)
        \alpha_{\mp,n-2l}
        \\[7pt]
        & \hspace{1cm}
        = \frac{\Gamma(n/2)}{\Gamma(n/2-k)}
        \Delta_{\mp, n-2k}
        + \Gamma(n/2)\sum_{l=1}^{k} 
        \frac{\alpha_{\mp,n-2l}}{\Gamma(n/2 - l + 1)},
      \end{align*}
      for all \(0<k<n/2\), where
      \[
        \widetilde{\alpha}_{\mp,n-2}
        = \frac{2d_{\mp}}{\nu}
        \int_{0}^{\infty}
        e^{-y}
        \left(1+\frac{2d_{\mp}y}{\nu}\right)^{-1}
        y^{n/2-1} dy
        = \alpha_{\mp,n-2}
      \]
      for all \(n\geq 1\). Hence, using mathematical induction one can obtain
      that for even \(n\)
      \begin{equation}
        \widetilde{\Delta}_{\mp, n}
        = \Gamma(n/2)\Delta_{\mp, 2}
        +\Gamma(n/2)
        \sum_{l=1}^{n/2-1}
        \frac{\alpha_{\mp,n-2l}}{\Gamma(n/2-l + 1)}
        = \Delta_{\mp, n}
      \end{equation}
      and if \(n\) is odd, then
      \begin{equation}
        \widetilde{\Delta}_{\mp, n}
        = \frac{\Gamma(n/2)}{\sqrt{\pi}}
        \Delta_{\mp, 1}
        +\Gamma(n/2)
        \sum_{l=1}^{(n-1)/2}
        \frac{\alpha_{\mp,n-2l}}{\Gamma(n/2-l+1)}
        =\Delta_{\mp, n}.
      \end{equation}
      Thus, \Cref{main_KL} follows.

      \bibliographystyle{ieeetr}
      \bibliography{total_variations.bib}
    \end{enumerate}
  \end{document}